\documentclass[article,12pt]{amsart}

\usepackage{amsmath}
\usepackage{amsfonts}
\usepackage{amssymb}
\usepackage{dsfont}
\usepackage{graphicx}
\usepackage{enumitem}
\usepackage{color}
\usepackage{array}
\usepackage{fullpage}
\usepackage{bbm}
\usepackage[round]{natbib}
\usepackage{MnSymbol}
\usepackage{upgreek }

\newtheorem{thm}{Theorem}[section]
\newtheorem{cor}[thm]{Corollary}
\newtheorem{lem}[thm]{Lemma}
\newtheorem{prop}[thm]{Proposition}

\theoremstyle{remark}
\newtheorem{rem}{Remark}[section]
\newtheorem*{exs}{Examples}

\newcommand{\veps}{\varepsilon}
\newcommand{\wh}{\widehat}
\newcommand{\whn}{\wh{\theta}_{n}}
\newcommand{\hsp}{\hspace{0.5cm}}
\newcommand{\T}{^{\, T}}
\newcommand{\diag}{\textnormal{diag}}

\newcommand{\cst}{C_{st}}
\renewcommand{\vec}{\textnormal{vec}}
\newcommand{\eqdef}{\, \overset{\Delta}{=}\, }
\newcommand{\tr}{\textnormal{tr}}
\renewcommand{\sp}{\textnormal{sp}}
\newcommand{\ind}{\mathbbm{1}}
\newcommand{\re}{\textnormal{Re}}

\newcommand{\dC}{\mathbb{C}}
\newcommand{\dE}{\mathbb{E}}
\newcommand{\dP}{\mathbb{P}}
\newcommand{\dR}{\mathbb{R}}
\newcommand{\dZ}{\mathbb{Z}}

\newcommand{\cF}{\mathcal{F}}
\newcommand{\cH}{\mathcal{H}}
\newcommand{\cL}{\mathcal{L}}
\newcommand{\cN}{\mathcal{N}}

\newcommand{\dd}{\, \mathrm{d}}

\newcommand{\cvgas}{~ \overset{a.s.}{\longrightarrow} ~}
\newcommand{\cvgp}{~ \overset{p}{\longrightarrow} ~}
\newcommand{\cvgd}{~ \overset{d}{\longrightarrow} ~}
\newcommand{\limn}{\lim_{n\, \rightarrow\, +\infty}}

\def\leq{\leqslant}
\def\geq{\geqslant}

 \numberwithin{equation}{section}
\allowdisplaybreaks

\begin{document}

\title[Consistency and asymptotic normality in a class of nearly unstable processes]{Consistency and asymptotic normality in a class of nearly unstable processes}

\author{Marie Badreau}
\address{Laboratoire Manceau de Mathématiques, Le Mans Université, Avenue Olivier Messiaen, 72085 LE MANS Cedex 09, France.}
\email{marie.badreau.etu@univ-lemans.fr}

\author{Fr\'ed\'eric Pro\"ia}
\address{Univ Angers, CNRS, LAREMA, SFR MATHSTIC, F-49000 Angers, France.}
\email{frederic.proia@univ-angers.fr}


\begin{abstract}
This paper deals with inference in a class of stable but nearly-unstable processes. Autoregressive processes are considered, in which the bridge between stability and instability is expressed by a time-varying companion matrix $A_{n}$ with spectral radius $\rho(A_{n}) < 1$ satisfying $\rho(A_{n}) \rightarrow 1$. This framework is particularly suitable to understand unit root issues by focusing on the inner boundary of the unit circle. Consistency is established for the empirical covariance and the OLS estimation together with asymptotic normality under appropriate hypotheses when $A$, the limit of $A_n$, has a real spectrum, and a particular case is deduced when $A$ also contains complex eigenvalues. The asymptotic process is integrated with either one unit root (located at 1 or $-1$), or even two unit roots located at 1 and $-1$. Finally, a set of simulations illustrate the asymptotic behavior of the OLS. The results are essentially proved by $L^2$ computations and the limit theory of triangular arrays of martingales.
\end{abstract}

\maketitle

\section{Introduction and motivation}
\label{Sec_Intro}

This paper is dedicated to the boundary between stationarity and integration in time series, which has long since proven to be crucial for practitioners, \textit{e.g.} in econometrics or finance. First of all, we refer the reader to \citet{BrockwellDavis91} for a large overview of linear time series. In the context of autoregressive processes (AR), it is well-known that the least squares (OLS) estimator is strongly consistent wherever its characteristic roots lie, see \citet{LaiWei83}, but with very different convergence rates and limit distributions whether the process is stationary or integrated, or even explosive. According to the terminology of \citet{Duflo97}, a stationary process will be described as \textit{stable} and an integrated process as \textit{unstable} in what follows. In the first case, the OLS estimator is $\sqrt{n}$-consistent with a Gaussian asymptotic behavior whereas in the other case, it is (at least) $n$-consistent with an asymmetrical asymptotic behavior that can be written as functionals of a standard Wiener process $(W(t),\, t \in [0, 1])$. Precisely, for example in the univariate case, in the stable (left-hand side) and unstable (right-hand side, $\theta = \pm 1$) cases, either
\begin{equation*}
\sqrt{n} \, \big( \whn - \theta \big) \cvgd \cN(0,1-\theta^2) \hsp \text{or} \hsp n\, \big( \whn-\theta \big) \cvgd  \frac{\pm \int_0^1 W(t) \dd W(t)}{\int_0^1 W^2(t) \dd t}
\end{equation*}
and these results are extended to the general case (see \textit{e.g.} the substantial work of \citet{ChanWei88}). This discontinuity has motivated numerous studies devoted to intermediate AR models aimed at bridging these two situations. Most of these studies involve the AR(1) process with time-varying coefficients, which is an easy-to-use relevant case to understand the underlying dynamics. Random coefficients have been suggested to bypass the difficulty: as it is explained in \citet{HorvathTrapani16}, thanks to a specific self-normalized WLS estimator of $\theta$, ``there is no \textit{unit root problem} in case of the RCA model". Indeed, the authors get the asymptotic normality irrespective of the average value of the coefficient, only the rate is affected. In a different context, the so-called \textit{volatility induced stationarity} detailed \textit{e.g.} in \citet{Ling04} and \citet{NielsenRahbek14}, enable to accommodate integration and stationarity in a double autoregression through stochastic volatility despite a unit root in the conditional mean of the process, see also \citet{Trapani21}. We refer the reader to all the references contained in those papers. An AR(1) with non-random but time-varying coefficients may also be written in the triangular form\footnote{Double indexing is customary to such representations: $X_{n,\, k}$ is the $k$-th observation of a time series of size $n$ (apart from the initial value). The triangular form of the process is $\{ X_{1,\, 0}, X_{1,\, 1} \}, \{ X_{2,\, 0}, X_{2,\, 1}, X_{2,\, 2} \}, \ldots, \{ X_{n,\, 0}, \ldots, X_{n,\, n} \}$.}
\begin{equation*}
\forall\, n \geq 1,\, \forall\, 1 \leq k \leq n, \hsp X_{n,\, k} = \theta_n\, X_{n,\, k-1} + \veps_k
\end{equation*}
where $(\veps_k)$ is a sequence of zero-mean perturbations (usually independent and identically distributed (i.i.d.) random variables, or differences of martingale) and $X_{n,\, 0}$ is an arbitrary initial value. In that case, the OLS estimator is obviously given by
\begin{equation}
\label{Intro_OLS}
\whn = \frac{\sum_{k=1}^n X_{n,\, k-1}\, X_{n,\, k}}{\sum_{k=1}^n X_{n,\, k-1}^{\, 2}}.
\end{equation}
To focus on the inner neighborhood of the unit root, the idea is to consider a sequence of coefficients that satisfies $\vert \theta_n \vert < 1$ for all $n \geq 1$ but $\vert \theta_n \vert \rightarrow 1$, so that a new model corresponds to each new observation, always stable but increasingly close to instability. \citet{ChanWei87} show that, letting $\theta_n = 1 - c/n$ and under appropriate assumptions, the OLS estimator \eqref{Intro_OLS} is $n$-consistent whether $c=0$ (which corresponds to the standard integrated AR(1) process) or $c>0$. Once self-normalized, it is shown in Thm. 1 that the estimation error has a non-Gaussian limit distribution $\cL(c)$ which is asymptotically $\cN(0, 1)$ as $c \rightarrow +\infty$ (see also Thm. 2 of \citet{Phillips87}). In this context, it is natural to conjecture that any rate faster than $n$ shall lead to the same conclusion. But slowing down the convergence of $\theta_n$ to 1 in order to make the aforementioned bridge, \citet{GiraitisPhillips06} and later \citet{PhillipsMagdalinos07} suggest to fix $\theta_n = 1 - c/v_n$ for $1 \ll v_n \ll n$. Among other results, they establish that
\begin{equation*}
\sqrt{n\, v_n}\, \big( \whn - \theta_n \big) \cvgd \cN(0,\, 2c)
\end{equation*}
as soon as $c > 0$, that is in the stable but nearly unstable case, and that the result also holds, under suitable assumptions, around the negative unit root with $\theta_n = -1 + c/v_n$. In the special case where $v_n = n^\alpha$ ($0 < \alpha < 1$), the rate is $n^{(1+\alpha)/2}$ and the authors note that the interval $\sqrt{n} \ll n^{(1+\alpha)/2} \ll n$ is covered but they emphasize that boundaries do not match. The asymptotic variance is overestimated for $\alpha \rightarrow 0^+$ ($2c$ instead of $2c - c^2$ when $\alpha=0$) whereas the Gaussian limit distribution for $\alpha \rightarrow 1^-$ is no longer Gaussian for $\alpha=1$. Note that the authors also consider the explosive case $\vert \theta_n \vert > 1$, whereas we will only focus here on the stable alternative to the unstable process. In the same vein, \citet{PhillipsLee15} develop a limit theory for nonstationary vector autoregressions with mixed roots in the vicinity of unity involving persistent and explosive components.

\smallskip

Now for any $p \geq 1$, consider the process generated according to the triangular form
\begin{equation}
\label{Intro_NearlyUnstableAR}
\forall\, n \geq 1,\, \forall\, 1 \leq k \leq n, \hsp X_{n,\, k} = \sum_{i=1}^{p} \theta_{n,\, i}\, X_{n,\, k-i} + \veps_k
\end{equation}
where $(\veps_k)$ is a sequence of zero-mean i.i.d. random variables with variance $\sigma^2 > 0$. In an equivalent way, the vector expression of this process is
\begin{equation}
\label{Intro_ModVAR}
\Phi_{n,\, k} = A_{n}\, \Phi_{n,\, k-1} + E_k
\end{equation}
where $E_k = (\veps_k, 0, \ldots, 0)\T$ is a $p$-vectorial noise, $\Phi_{n,\, k} = (X_{n,\, k}, \ldots, X_{n,\, k-p+1})\T$ and
\begin{equation*}
A_{n} = \begin{pmatrix}
\theta_{n,\, 1} & \theta_{n,\, 2} & \hdots & \theta_{n,\, p} \\
& I_{p-1} & & 0
\end{pmatrix}
\end{equation*}
is the associated $p \times p$ companion matrix. The initial value $\Phi_{n,\, 0}$ is supposed to have a finite moment of order 2 and to be independent of $(E_k)$. The OLS estimator of $\theta_n = (\theta_{n,\, 1}, \ldots, \theta_{n,\, p})\T$ is then given by
\begin{equation}
\label{Intro_OLSMulti}
\whn = S_{n,\, n-1}^{\, -1} \sum_{k=1}^{n} \Phi_{n,\, k-1}\, X_{n,\, k} \hsp \text{where} \hsp S_{n,\, n-1} = \sum_{k=0}^{n-1} \Phi_{n,\, k}\, \Phi_{n,\, k}\T.
\end{equation}
Note that we may add a small $\epsilon > 0$ to $S_{n,\, n-1}$ to avoid a useless invertibility assumption, without disturbing the asymptotic behavior.\footnote{\label{FN} To be rigorous, one should write $\wh{\theta}_{n,\, n}$ instead of $\whn$ to emphasize that the OLS is a function of $X_{n,\, 0}, \ldots, X_{n,\, n}$. Similarly, $S_{n}$ will be used for $S_{n,\, n}$  (and $T_{n}$ for $T_{n,\, n}$, etc.) to lighten the notation when no confusion can arise.} The $p$-dimensional process \eqref{Intro_ModVAR} is stable when $\rho(A_n) < 1$, that is, when the largest modulus of its eigenvalues is less than 1, see Def. 2.3.17 of \citet{Duflo97}. Since the eigenvalues of $A_n$ are the inverses of the roots of the complex polynomial $\Theta_n(z) = 1 - \theta_{n,\, 1}\, z - \ldots - \theta_{n,\, p}\, z^{\, p}$, this is equivalent to say that $\Theta_n(z) \neq 0$ for all $\vert z \vert \leq 1$. On the contrary, it is unstable when $\rho(A_n) = 1$. Along the same lines, we will thus consider that $\rho(A_n) < 1$ for all $n \geq 1$ but $\rho(A_n) \rightarrow 1$, which corresponds to a stable but nearly unstable AR$(p)$ process with time-varying coefficients. In this context and under suitable assumptions, \citet{Proia20} has established some moderate deviation principles for the empirical covariance and the OLS estimator that hold for any $p \geq 1$ (extending those of \citet{MiaoWangYang15} valid for $p=1$). In particular, the sequence
\begin{equation*}
\left( \frac{\sqrt{n}}{b_{n}\, (1 - \rho(A_{n}))^{1/2}}\, \big( \whn - \theta_n \big) \right)_{\!n\, \geq\, 1}
\end{equation*}
satisfies a large deviation principle with a speed $(b_n^{\, 2})$ and a rate function depending on the renormalized limit covariance of the process. Let us also mention the weak unit roots of \citet{Park03} containing more applications than ours, including faster rates of convergence for the nearly unit root and non-linear models, but in a more restricted setting (we will come back to this in due time), or the recent work of \citet{BuchmannChan13} who introduce a perturbation in the Jordan canonical form of the AR$(p)$ model (see Thm. 2.1) and get a set of convergences in a context close to ours (although more general). However, by directly dealing with the spectral radius of the companion matrix, the approach of this paper seems easier to interpret and we will see that \textit{in fine} different kind of results are obtained \textit{via} different technical tools. Now, to complete these deviations and to generalize the results of \citet{PhillipsMagdalinos07}, we aim at proving the consistency and asymptotic normality of the OLS estimator \eqref{Intro_OLSMulti}. The strategy remains the same, but the calculation steps will prove to be much trickier. In the second section, the assumptions and main results are provided together with comments and examples. The third section is dedicated to the technical proofs whereas the fourth section is the empirical part of the paper, containing simulations. A quick conclusion with considerations about further improvements ends the paper.

\section{Main results}
\label{Sec_Results}

First, let us start by describing three technical hypotheses that will be needed to achieve our goals. In particular, the first one is a matter of simplification of the reasonings since $A_n$ turns out to
be diagonalizable for a sufficiently large $n$, and that specificity will prove to be very useful. The second hypothesis is related to the number of unit roots in the asymptotic process (either $\pm 1$, or both 1 and $-1$). The third assumption characterizes our stable but nearly-unstable setting. The complex eigenvalues are sorted according to their modulus (in descending order), with ties broken by lexicographic order (also in descending order).

\begin{enumerate}[label=(H$_\arabic*$)]
\item \label{HypCA} \textit{Convergence of the companion matrix}. There exists a $p \times p$ matrix $A$ such that
\begin{equation*}
\limn A_{n} = A
\end{equation*}
with distinct eigenvalues $0\, <\, \vert \lambda_{p} \vert\, \leq\, \ldots\, \leq\, \vert \lambda_2 \vert\, \leq\, \vert \lambda_{1} \vert\, =\, \rho(A) = 1$, and the top-right element of $A$ is non-zero ($\theta_p \neq 0$). 
\item \label{HypUR} \textit{Number of unit roots}. Either \ref{HypUR1} or \ref{HypUR2} is true.
\begin{enumerate}[label=\:(H$_{2\arabic*}$)]
\item \label{HypUR1} There is exactly one unit root in $A$ ($\lambda_1 = \pm 1$ but $\vert \lambda_2 \vert < 1$ if $p \geq 2$).
\item \label{HypUR2} There is exactly two unit roots in $A$ ($\lambda_1 = 1$, $\lambda_2 = -1$ but $\vert \lambda_3 \vert < 1$ if $p \geq 3$).
\end{enumerate}
\item \label{HypSR} \textit{Spectral radius of the companion matrix}. The spectral radius of $A_n$ is given by
\begin{equation*}
\rho(A_{n}) = 1 - \frac{c}{v_n}
\end{equation*}
for some $c > 0$ and $1 \ll v_n \ll n$.
\end{enumerate}

Most of our results will be stated under \ref{HypUR1} but we will explain at the end of the section that, in fact, they are still valid under \ref{HypUR2} with only slight adjustments of the rates and limit behaviors. For readability purposes, the calculations will not be developed in that case but left to the reader since they follow exactly the same lines. A decisive argument in the technical part of the paper rests on the diagonalization of $A_n$, on which we will give more details in due course. But to summarize, we will explain that there exists $n_0 \geq 1$ such that, for all $n > n_0$, $A_n = P_n\, D_n\, P_n^{-1}$ where $D_n = \diag(\lambda_{n,\, 1},\, \ldots,\, \lambda_{n,\, p})$ contains the ordered distinct eigenvalues $1\, >\, \vert \lambda_{n,\, 1} \vert\, \geq\, \ldots\, \geq\, \vert \lambda_{n,\, p} \vert\, >\, 0$ of $A_n$ which converge to those of $A$ (see \ref{HypCA} above), and the basis of eigenvectors of $A_n$ can be written in the standardized form
\begin{equation}
\label{MatP}
P_n = \begin{pmatrix}
1 & 1 & \hdots & 1 \\
\frac{1}{\lambda_{n,\, 1}} & \frac{1}{\lambda_{n,\, 2}} & \hdots & \frac{1}{\lambda_{n,\, p}} \\
\vdots & \vdots & & \vdots \\
\frac{1}{\lambda_{n,\, 1}^{p-1}} & \frac{1}{\lambda_{n,\, 2}^{p-1}} & \hdots & \frac{1}{\lambda_{n,\, p}^{p-1}}
\end{pmatrix}.
\end{equation}
In addition, $P_n \rightarrow P$ and $P_n^{-1} \rightarrow P^{-1}$ whose entries of the first column are denoted by $\pi_{11}, \ldots, \pi_{p1}$. The $p \times p$ symmetric matrix $H$ and its standardized version $H_0$, that will play the role of the precision matrix in the asymptotic normality of the OLS estimate, are then defined as
\begin{equation}
\label{MatH}
H = \sigma^2\, \begin{pmatrix}
\frac{\pi_{11}^2}{2} & 0 & \hdots & 0 \\
0 & \frac{\pi_{21}^2}{1 - \lambda_2^2} & \hdots & \frac{\pi_{21} \pi_{p1}}{1 - \lambda_2 \lambda_p} \\
\vdots & \vdots & \ddots & \vdots \\
0 & \frac{\pi_{p1} \pi_{21}}{1 - \lambda_p \lambda_2} & \hdots & \frac{\pi_{p1}^2}{1 - \lambda_p^2}
\end{pmatrix} \eqdef \sigma^2\, H_0.
\end{equation}
From now on, to lighten the expressions we will rather use $\rho_n \eqdef \rho(A_{n})$. Also, for any $d \geq 1$, the usual $d \times d$ matrices
\begin{equation}
\label{MatK}
I_d = \begin{pmatrix} 1 & 0 & \hdots \\
 0 & \ddots & \\
 \vdots & & 1
\end{pmatrix} \hsp \text{and} \hsp K_d = \begin{pmatrix} 1 & 0 & \hdots \\
0 & 0 & \\
\vdots & & \ddots
\end{pmatrix}
\end{equation}
will be frequently encountered, both in statements and proofs. Similarly, the first vector of the canonical basis of $\dR^{d}$ will be denoted by $e_d = (1,0,\ldots,0)\T$. Finally, $f_n \asymp g_n$ will have the meaning that both $f_n = O(g_n)$ and $g_n = O(f_n)$.

\subsection{Causal representation and memory}

Under \ref{HypSR}, there exists a causal representation of $(\Phi_{n,\, k})$ given by
\begin{equation*}
\forall\, n \geq 1,\, \forall\, 1 \leq k \leq n, \hsp \Phi_{n,\, k} = \sum_{\ell\, \geq\, 0} A_n^\ell\, E_{k-\ell}
\end{equation*}
that directly leads to the autocovariance function
\begin{equation}
\label{ACV}
\forall\, h \geq 0, \hsp \Gamma_n(h) = \sigma^2 A_n^{h} \sum_{\ell\, \geq\, 0} A_n^\ell\, K_p\, (A_n\T)^\ell \hsp \text{and} \hsp \Gamma_n(-h) = \Gamma_n\T(h)
\end{equation}
provided that the noise has a finite moment of order 2. In particular, we have the following result related to the memory of the process.
\begin{prop}
\label{Prop_Mem}
Assume that \ref{HypCA}, \ref{HypUR} and \ref{HypSR} hold, and that $\dE[\veps_1^{\, 2}] = \sigma^2 < +\infty$. Then,
\begin{equation*}
\forall\, n > n_0, \hsp \left\Vert \sum_{h\, \in\, \dZ} \Gamma_n(h) \right\Vert < +\infty.
\end{equation*}
Besides, as $n \rightarrow +\infty$,
\begin{equation*}
\sum_{h\, \in\, \dZ} \Gamma_n(h) \asymp \left\{
\begin{array}{ll}
(1 - \rho_n)^{-2} & \textnormal{under \ref{HypUR1} with $\lambda_1=1$ or \ref{HypUR2}}, \\
1 & \textnormal{under \ref{HypUR1} with $\lambda_1=-1$.}
\end{array}
\right.
\end{equation*}
\end{prop}
\begin{proof}
See Section \ref{Sec_ProofCov}.
\end{proof}

In other terms, using the terminology of Sec. 1.3.1 of \citet{BeranEtAl13}, $(\Phi_{n,\, k})$ has a \textit{short memory} at fixed $n$ whereas, as $n$ tends to infinity, it turns to a \textit{long memory} process when $\lambda_1=1$ but keeps a short memory when $\lambda_1=-1$. This is especially clear in the univariate case for which
\begin{equation*}
\sum_{h\, \in\, \dZ} \gamma_n(h) = \frac{\sigma^2}{(1 - \theta_n)^2}
\end{equation*}
is an increasing function of $\theta_n$, minimal for $\theta_n \rightarrow -1^+$ and diverging to infinity for $\theta_n \rightarrow 1^-$. We are now ready to state the following result dedicated to the asymptotic behavior of the covariance matrix of the process.
\begin{prop}
\label{Prop_Cov}
Assume that \ref{HypCA}, \ref{HypUR1} and \ref{HypSR} hold, and that $\dE[\veps_1^{\, 2}] = \sigma^2 < +\infty$. Then, the empirical covariance matrix $S_{n,\, n} = S_n$ given in \eqref{Intro_OLSMulti} satisfies
\begin{equation*}
(1 - \rho_n)\, \frac{S_n}{n} \cvgp \frac{\sigma^2}{2}\, \vec^{-1}\big( (A^* \otimes A^*)\, e_{p^2} \big) \eqdef \Gamma \hsp \text{with} \hsp A^* = P\, K_p\, P^{-1}
\end{equation*}
where $\vec^{-1} : \dR^{p^2} \rightarrow \dR^{p \times p}$ is the vectorization inverse operator and the invertible matrix $P$ is the limit of $P_n$ in \eqref{MatP}. In addition,
\begin{equation*}
(1 - \rho_n) \left\Vert \frac{S_n}{n} - \Gamma_n \right\Vert \cvgp 0 \hsp \text{where} \hsp \Gamma_n = \Gamma_n(0) = \sigma^2 \sum_{\ell\, \geq\, 0} A_n^\ell\, K_p\, (A_n\T)^\ell
\end{equation*}
is the covariance of the stationary process (at fixed $n$).
\end{prop}
\begin{proof}
See Section \ref{Sec_ProofCov}.
\end{proof}

\begin{rem}
For the empirical covariance of the process to be consistent, it must be renormalized by $1-\rho_n$. In the stable case (stationary and ergodic), it is well-known that $S_n$ converges at rate $n$ whereas in the unstable case (integrated with one unit root, either positive or negative), $S_n$ converges at rate $n^2$. Let $p=1$ to simplify. Then, under suitable hypotheses, in the stable (left-hand side) and unstable (right-hand side) cases, either
\begin{equation*}
\frac{S_n}{n} \cvgas \gamma_0 \hsp \text{or} \hsp \frac{S_n}{n^2} \cvgd \sigma^2\! \int_0^1 W^2(t) \dd t
\end{equation*}
where $\gamma_0 > 0$ is the variance of the stationary process and $(W(t),\, t \in [0, 1])$ is a standard Wiener process. Proposition \ref{Prop_Cov} establishes the convergence of $S_{n,\, n}$ at rate $n/(1 - \rho_n) \propto n^{1+\alpha}$ for $v_n = n^{\alpha}$ and $0 < \alpha < 1$. One can see that our model bridges the stable and unstable cases in terms of empirical covariance even if the limit behaviors as $\alpha \rightarrow 0^+$ and $\alpha \rightarrow 1^-$ do not make a connection with $\alpha=0$ and $\alpha=1$.
\end{rem}

Even if it is instructive to control the behavior of the covariance matrix in a stationary process, that result shall \textit{not} help to prove the following theorems, dedicated to the OLS estimator, since the limit matrix is not invertible as soon as $p > 1$. Indeed a direct calculation shows that $(A^* \otimes A^*)\, e_{p^2}$ is a column vector of size $p^2$ with entries $\pm \pi^{\,2}_{11}$ and different rates of convergence must be used to achieve our objectives.

\subsection{OLS estimation}

 In view of the above, consider the $p \times p$ diagonal matrix defined as
\begin{equation}
\label{MatV}
V_n = \diag( 1 - \rho_n,\, 1,\, \ldots,\, 1)
\end{equation}
together with the $p^2 \times p^2$ matrix given, for $n > n_0$, by
\begin{equation}
\label{MatW}
W_n = ( V_n^{\, 1/2} P_n^{-1} ) \otimes ( V_n^{\, 1/2} P_n^{-1} )
\end{equation}
where $P_n$ is the (invertible) basis of eigenvectors of $A_n$.

\begin{thm}
\label{Thm_Consist}
Assume that \ref{HypCA}, \ref{HypUR1} and \ref{HypSR} hold, and that $\dE[\veps_1^{\, 2}] = \sigma^2 < +\infty$. Then we have the consistency
\begin{equation*}
\big\Vert \whn - \theta_n \big\Vert \cvgp 0
\end{equation*}
where $\whn$ is the OLS estimator \eqref{Intro_OLSMulti} in the nearly unstable AR$(p)$ process \eqref{Intro_NearlyUnstableAR}.
\end{thm}
\begin{proof}
See Section \ref{Sec_ProofConsist}.
\end{proof}

For the asymptotic normality, we need a slightly stronger hypothesis on the noise $(\veps_k)$. Let us assume from now on that there exists a moment of order strictly greater than 2.

\begin{thm}
\label{Thm_Norm}
Assume that \ref{HypCA}, \ref{HypUR1} and \ref{HypSR} hold, and that $\dE[\vert \veps_1 \vert^{\, 2+\nu}] = \eta_{\nu} < +\infty$ for some $\nu > 0$. Then, if the eigenvalues of $A$ are real, we have the asymptotic normality
\begin{equation*}
\sqrt{n}\, V_{n}^{-1/2} P_n^T\, \big( \whn - \theta_n \big) \cvgd \cN_p(0,\, H_0^{-1})
\end{equation*}
where $\whn$ is the OLS estimator \eqref{Intro_OLSMulti} in the nearly unstable AR$(p)$ process \eqref{Intro_NearlyUnstableAR}, the matrix of rates $V_n$ is given in \eqref{MatV}, $P_n$ is given in \eqref{MatP}, and the standardized positive definite precision matrix $H_0$ is given in \eqref{MatH}. If some eigenvalues of $A$ are complex, we have the asymptotic normality
\begin{equation*}
\sqrt{n\, v_n}\, \big\langle L_n,\, \whn - \theta_n \big\rangle \cvgd \cN(0,\, h_0^2)
\end{equation*}
where
\begin{equation*}
L_n = \left( 1, \frac{1}{\lambda_1 \rho_n}, \ldots, \frac{1}{(\lambda_1 \rho_n)^{p-1}} \right)\T \hsp \text{and} \hsp h_0^2 = \frac{2c}{\pi_{11}^2} > 0.
\end{equation*}
\end{thm}
\begin{proof}
See Section \ref{Sec_ProofNorm}.
\end{proof}

\begin{cor}
\label{Cor_Norm}
Under the same assumptions as in Theorem \ref{Thm_Norm} with $p=1$, we have the asymptotic normality
\begin{equation*}
\sqrt{n\, v_n}\, \big( \whn - \theta_n \big) \cvgd \cN(0,\, 2c).
\end{equation*}
\end{cor}
\begin{proof}
Just take $V_n = 1-\rho_n = c/v_n$, $P_n = 1$ and $H_0 = 1/2$ in Theorem \ref{Thm_Norm}.
\end{proof}

One can note that this is precisely the statement of Thm. 3.2(c) of \citet{PhillipsMagdalinos07} when the unit root is positive ($\lambda_1=1$). Nevertheless, our result also holds for a negative unit root ($\lambda_1 = -1$) whereas in the reference just mentioned, an extra symmetry argument is needed on the distribution of $(\veps_k)$. We may therefore say that Theorem \ref{Thm_Norm} reinforces and extends the result to the general near-stationary setting.

\begin{rem}
On the basis of Theorem \ref{Thm_Consist}, the triangle inequality directly implies that, under the same hypotheses,
\begin{equation}
\label{ConsistOLS}
\whn \cvgp \theta.
\end{equation}
But it is important to note that Theorem \ref{Thm_Norm} cannot lead to the corresponding asymptotic normality. To understand this, consider for example the univariate setting of Corollary \ref{Cor_Norm} where $\theta_n - \theta \propto 1/v_n$. Then,
\begin{equation*}
\sqrt{n\, v_n}\, \big\vert \theta_n - \theta \big\vert \propto \sqrt{\frac{n}{v_n}} ~ \longrightarrow ~ +\infty.
\end{equation*}
However, this is not surprising. Indeed, the limit AR process generated by $\theta$ is unstable. In that case, the OLS estimate is still consistent but not asymptotically normal.
\end{rem}

In terms of estimation, the mixing induced by $P_n^T$ is somewhat troublesome because it relies on the true and unknown eigenvalues of $A_n$. Since the eigenvalues of a square matrix depend continuously on its entries, see \textit{e.g.} Thm. 2.4.9.2 of \citet{HornJohnson92}, the consistency of $\whn$ would justify replacing $A_n$ by $\wh{A}_n$ before calculating its spectrum, so as to get consistent estimates of all quantities appearing in Theorem \ref{Thm_Norm} and build hypotheses tests. For simulations (Section \ref{Sec_Appli}), the following corollary will be used to illustrate the latter result.
\begin{cor}
\label{Cor_NormMarg}
Assume that the assumptions of Theorem \ref{Thm_Norm} hold. Then, for $\lambda_1 = \pm1$, we have the asymptotic distribution
\begin{equation*}
\frac{\pi_{11}^{\, 2}\, n\, v_n}{2\, c}\, \left[ \sum_{i=1}^{p} \frac{1}{(\lambda_1 \rho_n)^{\, i-1}}\, \big( \wh{\theta}_{n,\, i} - \theta_{n,\, i} \big) \right]^{\, 2} \cvgd \chi^2_1.
\end{equation*}
\end{cor}
\begin{proof}
This is in fact the first marginal convergence of the asymptotic normality combined with the continuous mapping theorem, \eqref{MatH} and the fact that the first column of $P_n$ is $(1, 1/\lambda_{n,\, 1}, \ldots, 1/\lambda_{n,\, 1}^{p-1})$.
\end{proof}

Before switching to the proofs of the results, some examples are provided. The goal is to get a clear insight into the structure of the fundamental matrix $H_0$ in simple cases.

\begin{exs} Let us consider the examples $p \in \{ 2, 3, 4 \}$ with $\lambda_1 = \pm 1$ (see Corollary \ref{Cor_Norm} for $p=1$). In fact, $H_0$ has a simple form deduced from the expressions
\begin{equation*}
\forall\, 1 \leq i \leq p, \hsp f_i = \prod_{j \neq i} (\lambda_j - \lambda_i).
\end{equation*}
\begin{itemize}
\item For $p=2$ and $\sp(A) = \{ \pm 1, \lambda_2 \}$ with $0 < \vert \lambda_2 \vert < 1$,
\begin{equation*}
H_0 = \begin{pmatrix}
\frac{\lambda_1^2}{2\, f_1^{\, 2}} & 0 \\
0 & \frac{\lambda_2^2}{f_2^{\, 2}\, (1-\lambda_2^2)}
\end{pmatrix}.
\end{equation*}
\item For $p=3$ and $\sp(A) = \{ \pm 1, \lambda_2, \lambda_3 \}$ with $0 < \vert \lambda_3 \vert \leq \vert \lambda_2 \vert < 1$ and $\lambda_2 \neq \lambda_3$,
\begin{equation*}
H_0 = \begin{pmatrix}
\frac{\lambda_1^4}{2\, f_1^{\, 2}} & 0 & 0 \\
0 & \frac{\lambda_2^4}{f_2^{\, 2}\, (1-\lambda_2^2)} & \frac{\lambda_2^2\, \lambda_3^2}{f_2\, f_3\, (1-\lambda_2 \lambda_3)} \\
0 & \frac{\lambda_3^2\, \lambda_2^2}{f_3\, f_2\, (1-\lambda_3 \lambda_2)} & \frac{\lambda_3^4}{f_3^{\, 2}\, (1-\lambda_3^2)}
\end{pmatrix}.
\end{equation*}
\item For $p=4$ and $\sp(A) = \{ \pm 1, \lambda_2, \lambda_3, \lambda_4 \}$ with $0 < \vert \lambda_4 \vert \leq \vert \lambda_3 \vert \leq \vert \lambda_2 \vert < 1$ and $\lambda_2 \neq \lambda_3 \neq \lambda_4$,
\begin{equation*}
H_0 = \begin{pmatrix}
\frac{\lambda_1^6}{2\, f_1^{\, 2}} & 0 & 0 & 0 \\
0 & \frac{\lambda_2^6}{f_2^{\, 2}\, (1-\lambda_2^2)} & \frac{\lambda_2^3\, \lambda_3^3}{f_2\, f_3\, (1-\lambda_2 \lambda_3)} & \frac{\lambda_2^3\, \lambda_4^3}{f_2\, f_4\, (1-\lambda_2 \lambda_4)} \\
0 & \frac{\lambda_3^3\, \lambda_2^3}{f_3\, f_2\, (1-\lambda_3 \lambda_2)} & \frac{\lambda_3^6}{f_3^{\, 2}\, (1-\lambda_3^2)} & \frac{\lambda_3^3\, \lambda_4^3}{f_3\, f_4\, (1-\lambda_3 \lambda_4)} \\
0 & \frac{\lambda_4^3\, \lambda_2^3}{f_4\, f_2\, (1-\lambda_4 \lambda_2)} & \frac{\lambda_4^3\, \lambda_3^3}{f_4\, f_3\, (1-\lambda_4 \lambda_3)} & \frac{\lambda_4^6}{f_4^{\, 2}\, (1-\lambda_4^2)}
\end{pmatrix}.
\end{equation*}

\end{itemize}
\end{exs}

As mentioned at the beginning of the section, Proposition \ref{Prop_Cov} and Theorems \ref{Thm_Consist} and \ref{Thm_Norm} are true under \ref{HypUR2}, \textit{i.e.} when $A$ contains the eigenvalues $1$ and $-1$ ordered this way (according to lexicographic descending order). That could correspond to a situation in which $\lambda_{n,\,1} = 1 - c/v_n$ and $\lambda_{n,\,2} = -1 + d/w_n$ with $c,\,d > 0$, $1 \ll v_n, w_n \ll n$ and $c/v_n \leq d/w_n$ (otherwise switching $\lambda_{n,\,1}$ and $\lambda_{n,\,2}$ in $P_n$ and $D_n$). More details will be given in due course (Remarks \ref{Rem_RatInvB}--\ref{Rem_ProofThm}). Let us now prove our results.

\section{Technical tools}
\label{Sec_Tools}

In this section, $\cst$ denotes a generic positive constant that is not necessarily identical from one line to another and we use the conventions $\sum_{\varnothing} = 0$ and $\Pi_{\varnothing} = 1$. Let us also define a fundamental matrix for the reasoning to come,
\begin{equation}
\label{MatB}
B_n = I_{p^2} - A_n \otimes A_n.
\end{equation}
In all the sequel, $\Vert \cdot \Vert$ will refer to the Frobenius norm $\Vert \cdot \Vert_F = \Vert \vec(\cdot) \Vert_2$ induced by the inner product $\llangle \cdot, \cdot \rrangle = \llangle \cdot, \cdot \rrangle_F = \langle \vec(\cdot), \vec(\cdot) \rangle$. The distinction is not made between matrices and vectors for which we simply have $\Vert \cdot \Vert_F = \Vert \cdot \Vert_2$ and $\llangle \cdot, \cdot \rrangle_F = \langle \cdot, \cdot \rangle$.

\subsection{Linear Algebra}
\label{Sec_LinAlg}

This section gathers all the linear algebra tools that we shall need in the proofs of our results.

\begin{lem}
\label{Lem_Diag}
Assume that \ref{HypCA} holds. Then, there exists $n_0 \geq 1$ such that, for all $n > n_0$, $A_n$ is diagonalizable in the form $A_n = P_n\, D_n\, P_n^{-1}$ with $D_n = \diag(\lambda_{n,\, 1},\, \ldots,\, \lambda_{n,\, p})$ containing ordered distinct eigenvalues $1\, >\, \vert \lambda_{n,\, 1} \vert\, \geq\, \ldots\, \geq\, \vert \lambda_{n,\, p} \vert\, >\, 0$. In addition, $\Vert P_n \Vert \leq \cst$ and $\Vert P_n^{-1} \Vert \leq \cst$.
\end{lem}
\begin{proof}
See Lem. 3.1 in \citet{Proia20}.
\end{proof}

\begin{lem}
\label{Lem_RowP1}
Assume that \ref{HypCA} holds. Then in the context of Lemma \ref{Lem_Diag}, for all $n > n_0$, the first row of $P_n^{-1}$ is real.
\end{lem}
\begin{proof}
This is obvious for $p=1$, so let $p > 1$. Suppose that, for $j \geq 2$, $\lambda_{n,\, j} \in \sp(A_n)$ is a complex eigenvalue so that, since $A_n$ is real, we must also have $\bar{\lambda}_{n,\, j} \in \sp(A_n)$. Taking $i=1$ in formula (3.4) of \citet{Proia20}, the first row of $P_n^{-1}$ is written
\begin{equation*}
\left(\frac{T_{n,\,1,\,k}\,\lambda_{n,\, 1}^{p-1}\, \prod_{j=2}^p \lambda_{n,\, j}}{\prod_{j=2}^p (\lambda_{n,\, j} - \lambda_{n,\, 1})} \right)_{1\, \leq\, k\, \leq\, p}
\end{equation*} \\
where $T_{n,\,1,\,k}$ can be retrieved from the relation
\begin{equation*}
\prod_{j=2}^p \left(X - \frac{1}{\lambda_{n,\, j}}\right) = T_{n,\,1,\,1} + T_{n,\,1,\,2}\, X + \ldots + T_{n,\,1,\,p}\, X^{p-1}
\end{equation*}
for real values of $X$. Combining such terms together with the fact that, for any $x \in \dR$ and for any $z \in \dC$,
\begin{equation*}
    (x-z)(x-\bar{z}) = x^2 + \vert z \vert^2  - 2\,\re(z)x \, \in \dR 
    \hsp \text{and} \hsp \lambda_{n,\, 1} = \pm \rho_n \in \dR, 
\end{equation*}
each entry of the first row of $P_n^{-1}$ is real.
\end{proof}

\begin{lem}
\label{Lem_ColP1}
Assume that \ref{HypCA} holds. Then, each element of the first column of $P^{-1}$ is non-zero.
\end{lem}
\begin{proof}
We remind that the entries of the first column of $P^{-1}$ are denoted by $\pi_{k1}$. Then, taking $j=1$ in the same formula as in the previous proof,
\begin{equation}\label{pik1}
\pi_{k1} = \frac{(-\lambda_k)^{p-1}}{\prod_{\ell\, \neq\, k} (\lambda_\ell - \lambda_k)}.
\end{equation}
Under \ref{HypCA}, all eigenvalues are distinct and non-zero, which concludes the proof.
\end{proof}

\begin{lem}
\label{Lem_InvB}
Assume that \ref{HypSR} holds. Then, for all $n \geq 1$, the matrix $B_n$ given in \eqref{MatB} is invertible.
\end{lem}
\begin{proof}
Indeed, since $\rho(A_n \otimes A_n) = \rho_n^{\, 2} < 1$ under \ref{HypSR}, the conclusion follows \textit{e.g.} from Cor. 5.6.16 of \citet{HornJohnson92}.
\end{proof}

\begin{lem}
\label{Lem_RatInvB}
Assume that \ref{HypCA}, \ref{HypUR1} and \ref{HypSR} hold. Then,
\begin{equation*}
\limn (1-\rho_n)\, B_n^{-1} = \frac{1}{2}\, (A^* \otimes A^*)
\end{equation*}
where $A^*$ is defined in Proposition \ref{Prop_Cov}.
\end{lem}
\begin{proof}
Once again the result is obvious for $p=1$, so let $p > 1$. Because $\rho(A_n \otimes A_n) < 1$ (see the proof of Lemma \ref{Lem_InvB}),
\begin{equation*}
B_n^{-1} = \sum_{k=0}^{+\infty} A_n^k \otimes A_n^k.
\end{equation*}
For $n > n_0$, it follows from simple linear algebra that
\begin{equation}
\label{DiagBn1}
B_n^{-1} = (P_n \otimes P_n)\, \left[ \sum_{k=0}^{+\infty} D_n^k \otimes D_n^k \right] (P_n^{-1} \otimes P_n^{-1}).
\end{equation}
From Lem. 3.1 of \citet{Proia20}, we know that $P_n \rightarrow P$ where the limit matrix $P$ is invertible under \ref{HypCA}, as a Vandermonde matrix with distinct entries. Thus we also have $P_n^{-1} \rightarrow P^{-1}$. It remains to observe that $D_n^k \otimes D_n^k$ is a $p^2 \times p^2$ diagonal matrix with top-left element $(\pm \rho_n)^{\, 2k}$ and all other elements given either by $(\pm \rho_n)^{\, k}\, \lambda^k_{n,\, i}$, depending on whether $\lambda_1 = 1$ or $\lambda_1 = -1$, or by $\lambda^k_{n,\, i}\, \lambda^k_{n,\, j}$ for $2 \leq i, j \leq p$. Under \ref{HypSR},
\begin{equation}
\label{CvgSumR2}
\limn (1-\rho_n)\, \sum_{k=0}^{+\infty} (\pm \rho_n)^{\, 2k} = \frac{1}{2}
\end{equation}
and, for $2 \leq i, j \leq p$,
\begin{equation}
\label{CvgSumRL}
\limn \sum_{k=0}^{+\infty} (\pm \rho_n)^{\, k}\, \lambda^k_{n,\, i} = \frac{1}{1 \mp \lambda_i} \hsp \text{and} \hsp \limn \sum_{k=0}^{+\infty} \lambda^k_{n,\, i}\, \lambda^k_{n,\, j} = \frac{1}{1 - \lambda_i \lambda_j}.
\end{equation}
Under \ref{HypUR1}, we can conclude that
\begin{equation}
\label{CvgBn1}
\limn (1-\rho_n)\, B_n^{-1} = \frac{1}{2}\, (P \otimes P)\, K_{p^2}\, (P^{-1} \otimes P^{-1})
\end{equation}
where $K_{p^2}$ is given in \eqref{MatK}. In a more `elegant' way,
\begin{equation*}
\frac{1}{2}\, (P \otimes P)\, K_{p^2}\, (P^{-1} \otimes P^{-1}) = \frac{1}{2}\, (P\, K_p\, P^{-1}) \otimes (P\, K_p\, P^{-1})
\end{equation*}
using the fact that $K_{p^2} = K_p \otimes K_p$, which concludes the proof.
\end{proof}

\begin{rem}
\label{Rem_RatInvB}
If \ref{HypUR2} is true instead of \ref{HypUR1}, the lemma still holds and the reasoning is similar but $K_{p^2}$ in \eqref{CvgBn1} must be replaced by the matrix of size $p^2 \times p^2$ having the diagonal block structure
\begin{equation*}
K^*_{p^2} = \begin{pmatrix}
K_p & & \\
 & \bar{K}_p & & \\
 & & 0 &
\end{pmatrix} \hsp \text{with} \hsp \bar{K}_p = \diag(0, 1, 0, \ldots, 0).
\end{equation*}
\end{rem}

\begin{lem}
\label{Lem_InvLam}
Assume that \ref{HypCA} and \ref{HypUR1} hold, and that $p>1$. Then, the $(p-1) \times (p-1)$ bottom-right block of $H_0$, which is the symmetric matrix given by
\begin{equation}
\label{MatLam}
\Lambda = 
\begin{pmatrix}
\frac{\pi_{21}^2}{1 - \lambda_2^2} & \hdots & \frac{\pi_{21} \pi_{p1}}{1 - \lambda_2 \lambda_p} \\
\vdots & \ddots & \vdots \\
\frac{\pi_{p1} \pi_{21}}{1 - \lambda_p \lambda_2} & \hdots & \frac{\pi_{p1}^2}{1 - \lambda_p^2}
\end{pmatrix},
\end{equation}
is invertible. 
\end{lem}
\begin{proof}
First under \ref{HypCA}, Lemma \ref{Lem_ColP1} implies $\pi_{k1} \neq 0$ for all $k$, so the result is obvious for $p=2$. Now let $p>2$ and denote by $\Lambda_k$ the $(k-1) \times (k-1)$ top-left submatrix of $\Lambda$ (so that $\Lambda_{p} = \Lambda$) and note that, although it is not a Cauchy matrix, it is closely related to it (see Sec. 0.9.12 of \citet{HornJohnson92}). Let us adopt the usual reasoning for calculating the determinant of such matrices. Multiplying the $(j-1)$-th column by $(1-\lambda_j \lambda_k) / (\pi_{j1}\pi_{k1})$ for $2 \leq j \leq k$, and subtracting the last column from each other column, it follows that
\begin{eqnarray*}
d_k & = & \prod_{j=2}^{k} \frac{\pi_{j1} \pi_{k1}}{(1 - \lambda_j \lambda_k)} \\
 & & \hsp \hsp \times \begin{vmatrix}
\frac{b_{2}(\lambda_2-\lambda_k)^2}{(1-\lambda_2^2)(1-\lambda_2 \lambda_k)} & \frac{b_{2}(\lambda_2-\lambda_k)(\lambda_3-\lambda_k)}{(1-\lambda_2\lambda_3)(1-\lambda_2\lambda_k)} 
& \hdots & 
\frac{b_{2}(\lambda_2-\lambda_k)(\lambda_{k-1}-\lambda_k)}{(1-\lambda_2 \lambda_{k-1}) (1-\lambda_2\lambda_k)} & * \\
\frac{b_{3}(\lambda_3-\lambda_k)(\lambda_2-\lambda_k)}{(1-\lambda_3\lambda_2)(1-\lambda_3\lambda_k)} &
\frac{b_{3}(\lambda_3-\lambda_k)^2}{(1-\lambda_3^2)(1-\lambda_3 \lambda_k)} & \hdots & 
\frac{b_{3}(\lambda_3-\lambda_k)(\lambda_{k-1} - \lambda_k)}{(1-\lambda_3 \lambda_{k-1})(1-\lambda_3 \lambda_k)} & * \\
\vdots & \vdots & \ddots & \vdots & \vdots \\
\frac{b_{k-1}(\lambda_{k-1} - \lambda_k)(\lambda_2-\lambda_k)}{(1-\lambda_{k-1}\lambda_2)(1-\lambda_{k-1} \lambda_k)} & 
\frac{b_{k-1}(\lambda_{k-1} - \lambda_k)(\lambda_3-\lambda_k)}{(1-\lambda_{k-1}\lambda_3)(1-\lambda_{k-1} \lambda_k)} & \hdots & 
\frac{b_{k-1}(\lambda_{k-1}-\lambda_k)^2}{(1-\lambda_{k-1}^2)(1-\lambda_{k-1} \lambda_k)} & * \\
0 & 0 & \hdots & 0 & 1
\end{vmatrix}
\end{eqnarray*}
where $d_k = \det(\Lambda_k)$, $b_i = (\pi_{i1}\pi_{j1}) / (\pi_{j1}\pi_{k1}) = \pi_{i1}/\pi_{k1}$ ($j$ stands for the column), and $*$ symbolizes useless entries for $d_k$. Thus, factorizing everything that can be factorized yields, for every $3 \leq k \leq p$,
\begin{equation*}
d_k = \frac{\pi_{k1}^2}{1-\lambda_k^2} \left(\prod_{j=2}^{k-1} \frac{ (\lambda_j - \lambda_k)^2}{(1 - \lambda_j \lambda_k)^2}\right) d_{k-1} \hsp \text{with} \hsp d_2 = \frac{\pi_{21}^2}{1-\lambda_2^2}.
\end{equation*}
The solution of the recurrence is given by
\begin{equation}
\label{DetLam}
\forall\, 2 \leq k \leq p, \hsp d_k = \left(\prod_{j=2}^{k} \frac{\pi_{j1}^2}{1-\lambda_j^2}\right) \left(\prod_{j = 2}^k \prod_{i\, <\, j} \frac{(\lambda_i - \lambda_j)^2}{(1 - \lambda_i \lambda_j)^2}\right).
\end{equation}
If \ref{HypCA} holds, then all eigenvalues are distinct and since $\pi_{j1}$ is non-zero for all $j$, $d_k$ is non-zero for all $k$. Taking $k = p$, the proof is now complete. If $A$ has only real eigenvalues, we can even show that, under \ref{HypCA}, $d_k>0$ for all $k$ which implies the positive definiteness of $\Lambda$ using Sylvester's criterion, see \textit{e.g.} Thm. 7.2.5 of \citet{HornJohnson92}.
\end{proof}
\begin{rem}
\label{Rem_InvLam}
Under \ref{HypUR2} and for $p>2$, we only have to consider the $(p-2) \times (p-2)$ bottom-right submatrix of $\Lambda$, which is obviously invertible (and positive definite when $A$ has only real eigenvalues), based on the above. 
\end{rem}

\subsection{Proof of Propositions \ref{Prop_Mem} and \ref{Prop_Cov}}
\label{Sec_ProofCov}

Let us start by noting that, under Lemma \ref{Lem_Diag} and \ref{HypSR}, for a fixed $n > n_0$,
\begin{eqnarray*}
\Vert \Gamma_n(0) \Vert & \leq & \cst \sum_{\ell\, \geq\, 0} \Vert A_n^\ell\, K_p\, (A_n\T)^\ell \Vert \\
 & \leq & \cst \sum_{\ell\, \geq\, 0} \Vert D_n^\ell \Vert^2\, \leq\, \frac{\cst}{1 - \rho_n^2}\, <\, +\infty
\end{eqnarray*}
where $\Gamma_n(\cdot)$ is the autocovariance function given in \eqref{ACV}. Thus, by the same reasoning,
\begin{eqnarray*}
\left\Vert \sum_{h\, \in\, \dZ} \Gamma_n(h) \right\Vert\, \leq\, \sum_{h\, \in\, \dZ} \Vert \Gamma_n(h) \Vert & \leq & \cst\, \Vert \Gamma_n(0) \Vert \sum_{h\, \in\, \dZ} \Vert A_n^{\vert h \vert} \Vert \\
 & \leq & \frac{\cst}{1 - \rho_n^2} \sum_{h\, \in\, \dZ} \Vert D_n^{\vert h \vert} \Vert\, \leq\, \frac{\cst}{(1-\rho_n)^2}\, <\, +\infty.
\end{eqnarray*}
Moreover, as we will see later in \eqref{CvgGam}, there exists a positive semi-definite matrix $\Gamma$ such that $(1 - \rho_n) \Gamma_n(0) \rightarrow \Gamma$. Hence, since $\Gamma_n(-h) = \Gamma_n\T(h)$,
\begin{equation}
\label{CvgSumP}
(1 - \rho_n)^2 \sum_{h\, \in\, \dZ} \Gamma_n(h) ~\longrightarrow~ A^* \Gamma + (A^* \Gamma)\T \eqdef \Gamma^+
\end{equation}
under \ref{HypUR1} with $\lambda_1 = 1$ or \ref{HypUR2}, where $A^* = P\, K_p\, P^{-1}$ and $K_p$ is given in \eqref{MatK}. The top-left element of $\Gamma^+$ is $\sigma^2 \pi_{11}^2 \neq 0$ (take \textit{e.g.} Lemma \ref{Lem_ColP1} with $k=1$). Under \ref{HypUR1} with $\lambda_1 = -1$, by the same calculations, it turns out that
\begin{equation}
\label{CvgSumM}
\sum_{h\, \in\, \dZ} \Gamma_n(h) ~\longrightarrow~ \Gamma^-
\end{equation}
where $\Gamma^-$ is also a non-zero matrix. The result is established by combining \eqref{CvgSumP} and \eqref{CvgSumM}.
\qed

\smallskip

As for the proof of Proposition \ref{Prop_Cov}, it will result from a chain of intermediate lemmas. See footnote \ref{FN} on page \pageref{FN} for notation.

\begin{lem}[Variance decomposition]
\label{Lem_DecompVar}
Assume that \ref{HypSR} holds. Then, for all $n \geq 1$,
\begin{eqnarray}
\label{DecompVar}
\vec\!\left( \frac{S_{n,\, n-1}}{n} \right) & = & B_n^{-1}\, \vec\!\left( \frac{T_n}{n} \right) + B_n^{-1}\, (I_p \otimes A_n)\, \vec\!\left( \frac{Z_n}{n} \right) \nonumber \\
 & & \hsp \hsp +~ B_n^{-1}\, (A_n \otimes I_p)\, \vec\!\left( \frac{Z_n\T}{n} \right) + B_n^{-1}\, \vec\!\left( \frac{L_n}{n} \right)
\end{eqnarray}
where $S_{n,\, n-1}$ is given in \eqref{Intro_OLSMulti}, $B_n$ is given in \eqref{MatB},
\begin{equation}
\label{DefTZL}
\hsp Z_n = \sum_{k=1}^{n} \Phi_{n,\, k-1}\, E_k\T, \hsp L_n = \sum_{k=1}^{n} E_k\, E_k\T
\end{equation}
and $T_n = \Phi_{n,\, 0}\, \Phi_{n,\, 0}\T - \Phi_{n,\, n}\, \Phi_{n,\, n}\T$.
\end{lem}
\begin{proof}
By direct calculation, we first obtain that, for all $n \geq 1$,
\begin{equation*}
\Phi_{n,\, k}\, \Phi_{n,\, k}\T = A_n\, \Phi_{n,\, k-1}\, \Phi_{n,\, k-1}\T\, A_n\T + A_n\, \Phi_{n,\, k-1}\, E_k\T + E_k\, \Phi_{n,\, k-1}\T\, A_n\T + E_k\, E_k\T.
\end{equation*}
Then, summing over $k$, it is not hard to see that
\begin{equation*}
S_{n,\, n-1} = T_n + A_n\, S_{n,\, n-1}\, A_n\T + A_n\, Z_n + Z_n\T A_n\T + L_n.
\end{equation*}
This is a generalized Sylvester equation w.r.t. $S_{n,\, n-1}$ having a unique solution since $\rho_n < 1$, see \textit{e.g.} Lem. 2.1(2) of \citet{JiangWei03}. We easily deduce that
\begin{equation*}
B_n\, \vec(S_{n,\, n-1}) = \vec(T_n) + (I_p \otimes A_n)\, \vec(Z_n) + (A_n \otimes I_p)\, \vec(Z_n\T) + \vec(L_n)
\end{equation*}
which, \textit{via} Lemma \ref{Lem_InvB}, gives the result.
\end{proof}

\begin{lem}[Isolated terms]
\label{Lem_IsolTerm}
Assume that \ref{HypCA} and \ref{HypSR} hold. Then,
\begin{equation*}
\frac{T_n}{n} \cvgp 0
\end{equation*}
where $T_n$ is given in \eqref{DefTZL}.
\end{lem}
\begin{proof}
First, since $\Phi_{n,\, 0}$ has a finite moment of order 2, Markov's inequality directly gives
\begin{equation}
\label{IsolTerm_CvgT0}
\frac{\Vert \Phi_{n,\, 0}\, \Phi_{n,\, 0}\T \Vert}{n} \cvgp 0.
\end{equation}
Then, by the triangle inequality and exploiting Lemma \ref{Lem_Diag}, for all $n > n_0$,
\begin{eqnarray*}
\Vert \Phi_{n,\, n}\, \Phi_{n,\, n}\T \Vert & \leq & 2\, \Vert A_n^n \Vert^2\, \Vert \Phi_{n,\, 0} \Vert^2 + 2\, \left\Vert \sum_{k=0}^{n-1} A_n^k\, E_{n-k} \right\Vert^2 \nonumber \\
 & \leq & \cst\, \rho_n^{\, 2n}\, \Vert \Phi_{n,\, 0} \Vert^2 + 2\, \left\Vert \sum_{k=0}^{n-1} A_n^k\, E_{n-k} \right\Vert^2 ~ \eqdef ~ T_{1,\, n} + T_{2,\, n}.
\end{eqnarray*}
Note that, according to \ref{HypSR},
\begin{eqnarray*}
\rho_n^{\, 2n} & = & \left( 1 - \frac{c}{v_n} \right)^{2n} \nonumber \\
 & = & \exp\left( \frac{-2c\, n}{v_n} + o\left( \frac{n}{v_n} \right) \right) ~ \longrightarrow ~ 0.
\end{eqnarray*}
We can conclude that $\dE[\vert T_{1,\, n} \vert] \rightarrow 0$, and thus
\begin{equation}
\label{IsolTerm_CvgT1}
\frac{T_{1,\, n}}{n} \cvgp 0.
\end{equation}
Moreover,
\begin{eqnarray*}
\left\Vert \sum_{k=0}^{n-1} A_n^k\, E_{n-k} \right\Vert^2 & = & \sum_{k=0}^{n-1} \sum_{\ell=0}^{n-1} \llangle P_n\, D_n^k\, P_n^{-1}\, E_{n-k},\, P_n\, D_n^\ell\, P_n^{-1}\, E_{n-\ell} \rrangle \nonumber \\
 & = & \sum_{k=0}^{n-1} \sum_{\ell=0}^{n-1} \tr\big( P_n\, D_n^\ell\, P_n^{-1}\, E_{n-\ell}\, E_{n-k}\T\, (P_n^{-1})\T\, D_n^k\, P_n\T \big).
\end{eqnarray*}
Consequently,
\begin{eqnarray*}
\dE[\vert T_{2,\, n} \vert] & = & 2\, \sigma^2\, \sum_{k=0}^{n-1} \tr\big( P_n\, D_n^k\, P_n^{-1}\, K_p\, (P_n^{-1})\T\, D_n^k\, P_n\T \big) \nonumber \\
 & = & \cst\, \sum_{k=0}^{n-1} \Vert P_n\, D_n^k\, P_n^{-1}\, K_p \Vert^2
\end{eqnarray*}
where $K_p$ is given in \eqref{MatK}, and obviously satisfies $K_p = K_p\T$ and $K_p^{\, 2} = K_p$. Hence, using the same arguments as before,
\begin{eqnarray}
\label{IsolTerm_EstT2}
\dE[\vert T_{2,\, n} \vert] & \leq & \cst\, \sum_{k=0}^{n-1} \rho_n^{\, 2k} \nonumber \\
 & = & \cst\, v_n\, \frac{1 - \exp\left( \frac{-2c\, n}{v_n} + o\left( \frac{n}{v_n} \right) \right)}{2c - \frac{c^2}{v_n}} ~ = ~ O(v_n).
\end{eqnarray}
Since $v_n = o(n)$, that leads to
\begin{equation}
\label{IsolTerm_CvgT2}
\frac{T_{2,\, n}}{n} \cvgp 0.
\end{equation}
The combination of \eqref{IsolTerm_CvgT0}, \eqref{IsolTerm_CvgT1} and \eqref{IsolTerm_CvgT2} concludes the proof.
\end{proof}

\begin{lem}[Martingale terms]
\label{Lem_MartTerm}
Assume that \ref{HypCA} and \ref{HypSR} hold. Then,
\begin{equation*}
\frac{Z_n}{n} \cvgp 0
\end{equation*}
where $Z_n$ is given in \eqref{DefTZL}.
\end{lem}
\begin{proof}
By the triangle inequality,
\begin{eqnarray*}
\Vert Z_n \Vert & \leq & \Vert \Phi_{n,\, 0}\, E_1\T \Vert + \left\Vert \sum_{k=2}^n A_n^{k-1}\, \Phi_{n,\, 0}\, E_k\T \right\Vert + \left\Vert \sum_{k=2}^n \sum_{\ell=0}^{k-2} A_n^{\ell}\, E_{k-1-\ell}\, E_k\T \right\Vert \nonumber \\
 & \eqdef & T_{0,\, n} + T_{1,\, n} + T_{2,\, n}.
\end{eqnarray*}
Clearly, by independence and existence of moments of order 2,
\begin{equation}
\label{MartTerm_CvgT0}
\frac{T_{0,\, n}}{n} \cvgp 0.
\end{equation}
Then, a direct calculation shows that
\begin{eqnarray*}
T_{1,\, n}^{\, 2} & = & \sum_{k=2}^n \sum_{\ell=2}^{n} \llangle A_n^{k-1}\, \Phi_{n,\, 0}\, E_k\T,\, A_n^{\ell-1}\, \Phi_{n,\, 0}\, E_\ell\T \rrangle \nonumber \\
 & = & \sum_{k=2}^n \sum_{\ell=2}^{n} \tr\big(\Phi_{n,\, 0}\T\, (A_n\T)^{\ell-1}\, A_n^{k-1}\, \Phi_{n,\, 0}\, E_k\T E_\ell \big)
\end{eqnarray*}
exploiting the cyclic property of the trace. Hence, for any $n > n_0$,
\begin{eqnarray*}
\dE[T_{1,\, n}^{\, 2}] & = & \sum_{k=2}^n \sum_{\ell=2}^{n} \tr\big( \dE\big[ \Phi_{n,\, 0}\T\, (A_n\T)^{\ell-1}\, A_n^{k-1}\, \Phi_{n,\, 0}  \big]\, \dE[ E_k\T E_\ell ] \big) \nonumber \\
 & = & \sigma^2\, \sum_{k=2}^n \dE\big[ \tr\big( \Phi_{n,\, 0}\T\, (A_n^{k-1})\T\, A_n^{k-1}\, \Phi_{n,\, 0} \big) \big] \nonumber \\
 & = & \sigma^2\, \sum_{k=2}^n \dE\big[ \Vert A_n^{k-1}\, \Phi_{n,\, 0} \Vert^2 \big] \nonumber \\
 & \leq & \cst\, \sum_{k=2}^n \rho_n^{\, 2(k-1)}\, \dE[ \Vert \Phi_{n,\, 0} \Vert^2 ] ~ = ~ O(v_n)
\end{eqnarray*}
using \eqref{IsolTerm_EstT2}, Lemma \ref{Lem_Diag} and the second-order moments of $\Phi_{n,\, 0}$. Since $v_n = o(n)$,
\begin{equation}
\label{MartTerm_CvgT1}
\frac{T_{1,\, n}^{\, 2}}{n} \cvgp 0 \hsp \text{so that} \hsp \frac{T_{1,\, n}}{n} \cvgp 0.
\end{equation}
To handle the last term of the decomposition, let, for any $k \geq 1$,
\begin{equation*}
U_k = \sum_{j=1}^k A_n^{k-j}\, E_j
\end{equation*}
and note that
\begin{equation}
\label{MartTerm_EspUkUl}
\forall\, k \neq \ell, \hsp \dE\big[ U_{\ell-1}\T\, U_{k-1}\, E_k\T E_\ell \big] = 0.
\end{equation}
Now,
\begin{eqnarray*}
T_{2,\, n}^{\, 2} & = & \left\Vert \sum_{k=2}^n U_{k-1}\, E_k\T \right\Vert^2 \nonumber \\
 & = & \sum_{k=2}^n \sum_{\ell=2}^n \llangle U_{k-1}\, E_k\T,\, U_{\ell-1}\, E_\ell\T \rrangle ~ = ~ \sum_{k=2}^n \sum_{\ell=2}^n \tr\big( U_{\ell-1}\T\, U_{k-1}\, E_k\T E_\ell \big)
\end{eqnarray*}
exploiting again the cyclic property of the trace. Thanks to relations \eqref{MartTerm_EspUkUl}, it follows that, for $n > n_0$,
\begin{eqnarray*}
\dE[T_{2,\, n}^{\, 2}] & = & \sigma^2\, \sum_{k=2}^n \tr\big( \dE\big[ U_{k-1}\T\, U_{k-1} \big] \big) \nonumber \\
 & = & \sigma^2\, \sum_{k=2}^n \sum_{j=1}^{k-1} \sum_{i=1}^{k-1} \tr\big( (A_n\T)^{k-1-i}\, A_n^{k-1-j}\, \dE[E_j\, E_i\T] \big) \nonumber \\
 & = & \sigma^4\, \sum_{k=2}^n \sum_{j=1}^{k-1} \tr\big( (A_n\T)^{k-1-j}\, A_n^{k-1-j}\, K_p \big) = \sigma^4\, \sum_{k=2}^n \sum_{\ell=0}^{k-2} \Vert A_n^\ell\, K_p \Vert^2 ~ = ~ O(n v_n)
\end{eqnarray*}
where $K_p$ is given in \eqref{MatK}, and where the rate $v_n$ is obtained \textit{via} the same lines as \eqref{IsolTerm_EstT2}. Consequently,
\begin{equation}
\label{MartTerm_CvgT2}
\frac{T_{2,\, n}^{\, 2}}{n^2} \cvgp 0 \hsp \text{so that} \hsp \frac{T_{2,\, n}}{n} \cvgp 0.
\end{equation}
The combination of \eqref{MartTerm_CvgT0}, \eqref{MartTerm_CvgT1} and \eqref{MartTerm_CvgT2} concludes the proof.
\end{proof}
Let us now return to the proof of Proposition \ref{Prop_Cov}. By the law of large numbers, the last term of \eqref{DecompVar} is such that
\begin{equation}
\label{CvgL}
\frac{L_n}{n} \cvgp \sigma^2\, K_p
\end{equation}
where $K_p$ is given in \eqref{MatK}. By combining Lemmas \ref{Lem_RatInvB}, \ref{Lem_DecompVar}, \ref{Lem_IsolTerm} and \ref{Lem_MartTerm}, convergence \eqref{CvgL} and hypotheses \ref{HypCA}, \ref{HypUR1} and \ref{HypSR}, the first result is proved. It is important to note that $A^* = P\, K_p\, P^{-1}$ is a real matrix (see Lemma \ref{Lem_RowP1}), so that the limit of the renormalized empirical covariance is obviously real itself. For the second result, we have
\begin{eqnarray}
\label{CvgGam}
(1-\rho_n)\, \vec(\Gamma_n) & = & \sigma^2\, (1-\rho_n) \left[ \sum_{k=0}^{+\infty} A_n^k \otimes A_n^k \right] \vec(K_p) \nonumber \\
& = & \sigma^2 \left( (1-\rho_n)\, B_n^{-1} \right) e_{p^2} \nonumber \\
& \longrightarrow & \frac{\sigma^2}{2}\, (A^* \otimes A^*) e_{p^2} ~=~ \vec(\Gamma)
\end{eqnarray}
from Lemma \ref{Lem_RatInvB}, where $\Gamma$ is defined in Proposition \ref{Prop_Cov} and noting that $\vec(K_p) = e_{p^2}$. Finally,
\begin{equation*}
(1 - \rho_n) \left\Vert \frac{S_n}{n} - \Gamma_n \right\Vert\, \leq\, \left\Vert (1 - \rho_n)\, \frac{S_n}{n} - \Gamma \right\Vert + \big\Vert (1 - \rho_n)\, \Gamma_n - \Gamma \big\Vert \cvgp 0
\end{equation*}
exploiting the first part of the proposition and \eqref{CvgGam}. The proof is now complete. \qed

\begin{rem}
\label{Rem_ProofProp}
If \ref{HypUR2} is true instead of \ref{HypUR1}, the results of the proposition still hold provided that the limit is adjusted (see Remark \ref{Rem_RatInvB}).
\end{rem}

\subsection{Proof of Theorem \ref{Thm_Consist}}
\label{Sec_ProofConsist}

Note that for all $n > n_0$, under \ref{HypCA} and like in \eqref{DiagBn1},
\begin{equation*}
W_n\, B_n^{-1} = (V_n^{\, 1/2} \otimes V_n^{\, 1/2})\, \left[ \sum_{k=0}^{+\infty} D_n^k \otimes D_n^k \right] (P_n^{-1} \otimes P_n^{-1})
\end{equation*}
where $W_n$ comes from \eqref{MatW}. The matrix of rates is structured in $p$ blocks as follows,
\begin{equation*}
V_n^{\, 1/2} \otimes V_n^{\, 1/2} = \diag\big( \sqrt{1 - \rho_n}\, V_n^{\, 1/2},\, V_n^{\, 1/2},\, \ldots,\, V_n^{\, 1/2} \big).
\end{equation*}
Thus, taking over the reasoning of \eqref{CvgSumR2} and \eqref{CvgSumRL}, we can show that
\begin{equation}
\label{LimWB}
\limn W_n\, B_n^{-1} = \diag\left( \frac{K_p}{2},\, \Delta_2,\, \ldots,\, \Delta_p \right) (P^{-1} \otimes P^{-1}) \eqdef H^*
\end{equation}
using \ref{HypUR1} and \ref{HypSR}, where $K_p$ is given in \eqref{MatK} and
\begin{equation*}
\forall\, 2 \leq i \leq p, \hsp \Delta_i = \diag\left( 0,\, \frac{1}{1-\lambda_i\, \lambda_2},\, \ldots,\, \frac{1}{1-\lambda_i\, \lambda_p}  \right).
\end{equation*}
That leads, together with Lemmas \ref{Lem_DecompVar}, \ref{Lem_IsolTerm}, \ref{Lem_MartTerm} and formula \eqref{CvgL}, to the convergence
\begin{equation*}
W_n\, \vec\!\left( \frac{S_{n,\, n-1}}{n} \right) \cvgp \sigma^2 H^* e_{p^2}.
\end{equation*}
Since
\begin{equation*}
W_n\, \vec\!\left( \frac{S_{n,\, n-1}}{n} \right) = \vec\!\left( V_n^{\, 1/2}\, P_n^{-1}\, \frac{S_{n,\, n-1}}{n}\, (P_n^{-1})\T\, V_n^{\, 1/2} \right),
\end{equation*}
we obtain that
\begin{equation}
\label{CvgSnRenorm}
V_n^{\, 1/2}\, P_n^{-1}\, \frac{S_{n,\, n-1}}{n}\, (P_n^{-1})\T\, V_n^{\, 1/2} \cvgp \vec^{-1}(\sigma^2 H^* e_{p^2})
\end{equation}
where $\vec^{-1} : \dC^{p^2} \rightarrow \dC^{p \times p}$ is the vectorization inverse operator. A straightforward calculation shows that $\vec^{-1}(\sigma^2 H^* e_{p^2}) = H$ where $H$ is given in \eqref{MatH}, and this limit is invertible. Indeed, on the one hand,
\begin{equation*}
\pi_{11} = \frac{(-\lambda_1)^{p-1}}{\prod_{j=2}^p (\lambda_j - \lambda_1)}
\end{equation*}
from formula \eqref{pik1} with $k=1$, and we know from Lemma \ref{Lem_RowP1} that $\pi_{11}$ is real and non-zero (positive). On the other hand, Lemma \ref{Lem_InvLam} implies that $\Lambda$ is invertible so that $\det(H)$ is obviously non-zero. In addition, the estimation error satisfies
\begin{eqnarray}
\label{EstErr}
\whn - \theta_n & = & (P_n\T)^{-1}\, V_n^{\, 1/2}\, \left( V_n^{\, 1/2}\, P_n^{-1}\, \frac{S_{n,\, n-1}}{n}\, (P_n^{-1})\T\, V_n^{\, 1/2} \right)^{-1} \nonumber \\
 & & \hsp \hsp \times \left( V_n^{\, 1/2}\, P_n^{-1} \frac{Z_n}{n}\, e_{p} \right)
\end{eqnarray}
where in this expression, $Z_n$ comes from \eqref{DefTZL}. By the reasoning above and Lemmas \ref{Lem_Diag} and \ref{Lem_MartTerm},
\begin{equation*}
\big\Vert \whn - \theta_n \big\Vert\, \leq\, \cst \left\Vert \left( V_n^{\, 1/2}\, P_n^{-1}\, \frac{S_{n,\, n-1}}{n}\, (P_n^{-1})\T\, V_n^{\, 1/2} \right)^{-1} \right\Vert\, \left\Vert \frac{Z_n}{n} \right\Vert \cvgp 0.
\end{equation*}
That establishes the consistency of the OLS estimator. \qed

\subsection{Proof of Theorem \ref{Thm_Norm}}
\label{Sec_ProofNorm}

Suppose first that the eigenvalues of $A$ are real, and consider the filtration
\begin{equation*}
\forall\, n \geq 1,\, \forall\, 1 \leq k \leq n, \hsp \cF_{n,\, k} = \sigma(\Phi_{n,\, 0},\, \veps_1,\, \ldots,\, \veps_k).
\end{equation*}
For all $a \in \dR^p \backslash \{ 0 \}$, let also
\begin{equation}
\label{Mart}
m_{n,\, k}^{(a)} = a\T\, V_n^{\, 1/2}\, P_n^{-1}\, \Phi_{n,\, k-1}\, \veps_k \hsp \text{and} \hsp M_n^{(a)} = \sum_{k=1}^n m_{n,\, k}^{(a)}
\end{equation}
the way it appears in the right-hand side of \eqref{EstErr}. The sequence $(m_{n,\, k}^{(a)})$  is clearly a scalar martingale difference array w.r.t. $\cF_{n,\, k}$ at fixed $n$ and for $1 \leq k \leq n$. The predictable quadratic variation of $M_n^{(a)}$ is
\begin{eqnarray*}
\langle M^{(a)} \rangle_n & = & \sum_{k=1}^n \dE[ (m_{n,\, k}^{(a)})^2\, \vert\, \cF_{n,\, k-1} ] \\
 & = & \sigma^2\, a\T\, V_n^{\, 1/2}\, P_n^{-1} \left[ \sum_{k=1}^n \Phi_{n,\, k-1}\, \Phi_{n,\, k-1}\T \right] (P_n^{-1})\T\, V_n^{\, 1/2}\, a
\end{eqnarray*}
since $(\veps_k)$ is a white noise and $\Phi_{n,\, k-1}$ is obviously $\cF_{n,\, k-1}$-measurable. Together with \eqref{CvgSnRenorm} and the definition of $S_{n,\, n-1}$ in \eqref{Intro_OLSMulti}, that leads to the convergence
\begin{equation}
\label{CvgCroch}
\frac{\langle M^{(a)} \rangle_n}{n} \cvgp \sigma^2\, a\T H a\, >\, 0
\end{equation}
where $H$ is the covariance matrix given in \eqref{MatH}, positive definite as shown in the proof of Lemma \ref{Lem_InvLam} when $A$ has only real eigenvalues. To apply the central limit theorem for arrays of martingales, see \textit{e.g.} Thm. 1 of Sec. 8 in \citet{Pollard84}, it remains to show that the Lindeberg's condition is satisfied, in other words that
\begin{equation}
\label{Lindeberg}
\forall\, \epsilon > 0, \hsp \frac{1}{n} \sum_{k=1}^n \dE\big[ (m_{n,\, k}^{(a)})^2\, \ind_{ \{ \vert m_{n,\, k}^{(a)} \vert\, >\, \epsilon \sqrt{n} \} }\, \vert\, \cF_{n,\, k-1} \big] \cvgp 0.
\end{equation}
To prove \eqref{Lindeberg}, we can first see that, for any $1 \leq k \leq n$,
\begin{eqnarray*}
\dE\big[ (m_{n,\, k}^{(a)})^2\, \ind_{ \{ \vert m_{n,\, k}^{(a)} \vert\, >\, \epsilon \sqrt{n} \} }\, \vert\, \cF_{n,\, k-1} \big] & = & T^{(a)}_{n,\, k-1}\, \xi_{n,\, k} \\
& \leq & T^{(a)}_{n,\, k-1}\, \sup_{1\, \leq\, k\, \leq\, n} \xi_{n,\, k}
\end{eqnarray*}
where $T^{(a)}_{n,\, k-1} \eqdef a\T\, V_n^{\, 1/2}\, P_n^{-1}\, \Phi_{n,\, k-1}\, \Phi_{n,\, k-1}\T (P_n^{-1})\T\, V_n^{\, 1/2}\, a\, >\, 0$ and with, by H\"older's and Markov's inequalities,
\begin{eqnarray*}
\xi_{n,\, k} & \eqdef & \dE\big[ \veps_k^{\, 2}\, \ind_{ \{ \vert m_{n,\, k}^{(a)} \vert\, >\, \epsilon \sqrt{n} \} }\, \vert\, \cF_{n,\, k-1} \big] \\
 & \leq & \eta_{\nu}^{\, 2/(2+\nu)}\, \dP\big( (m_{n,\, k}^{(a)})^2\, >\, \epsilon^{\, 2} n\, \vert\, \cF_{n,\, k-1} \big)^{\nu/(2+\nu)} \leq~ \cst \left( \frac{T^{(a)}_{n,\, k-1}}{n} \right)^{\! \nu/(2+\nu)}
\end{eqnarray*}
in which $\eta_{\nu}$ is the moment of order $2+\nu$ of $(\veps_k)$, as defined in the statement of the theorem. But choosing $k=n$, \eqref{CvgCroch} directly entails
\begin{equation}
\label{CvgSupXi}
\xi_{n,\, n} \cvgp 0 \hsp \text{so that} \hsp \xi^{\, \sharp}_{n,\, n} \eqdef \sup_{1\, \leq\, k\, \leq\, n} \xi_{n,\, k} \cvgp 0
\end{equation}
according to Lem. 1.3.20 of \citet{Duflo97} adapted to convergence in probability. Finally,
\begin{equation*}
\frac{1}{n} \sum_{k=1}^n \dE\big[ (m_{n,\, k}^{(a)})^2\, \ind_{ \{ \vert m_{n,\, k}^{(a)} \vert\, >\, \epsilon \sqrt{n} \} }\, \vert\, \cF_{n,\, k-1} \big]\, \leq\, \left( \frac{1}{n} \sum_{k=1}^n T^{(a)}_{n,\, k-1} \right) \xi^{\, \sharp}_{n,\, n}
\end{equation*}
and, \textit{via} \eqref{CvgSnRenorm} and \eqref{CvgSupXi}, Lindeberg's condition \eqref{Lindeberg} is satisfied. Together with \eqref{CvgCroch}, this is sufficient to establish that
\begin{equation*}
\forall\, a \in \dR^p \backslash \{ 0 \}, \hsp \frac{M_n^{(a)}}{\sqrt{n}} \cvgd \cN(0,\, \sigma^2\, a\T H a).
\end{equation*}
Because $a$ is arbitrary, we can now make use of the Cramér-Wold device to get the $p$-vectorial convergence
\begin{equation}
\label{CvgMart}
\frac{M_n}{\sqrt{n}} \cvgd \cN_p(0,\, \sigma^2\, H)
\end{equation}
with
\begin{equation*}
M_n = V_n^{\, 1/2}\, P_n^{-1}\, \sum_{k=1}^n \Phi_{n,\, k-1}\, \veps_k.
\end{equation*}
Coming back to \eqref{CvgSnRenorm} and \eqref{EstErr}, and applying Slutsky's lemma, the proof is complete for the real case. If the spectrum of $A$ contains complex eigenvalues, we shall exploit the fact that the first row of $P_n^{-1}$ (say, $\ell_n\T$) still contains real entries, as it is established in Lemma \ref{Lem_RowP1}. The scalar array of $\cF_{n,\, k}$-martingale to use is now
\begin{equation*}
\forall\, n \geq 1, \hsp M_n = \sum_{k=1}^n m_{n,\, k}
\end{equation*}
with
\begin{equation*}
\forall\, 1 \leq k \leq n, \hsp m_{n,\, k} = \sqrt{1-\rho_n}\, \ell_n\T\, \Phi_{n,\, k-1}\, \veps_k.
\end{equation*}

\smallskip

\noindent Similarly, the predictable quadratic variation of $M_n$ is
\begin{equation*}
\langle M \rangle_n = \sigma^2\, (1-\rho_n)\, \ell_n\T\, \left[ \sum_{k=1}^n \Phi_{n,\, k-1}\, \Phi_{n,\, k-1}\T \right] \ell_n.
\end{equation*}
From \eqref{CvgSnRenorm},
\begin{equation}
\label{CvgMartCompl}
\frac{\langle M \rangle_n}{n} \cvgp \sigma^4\, \frac{\pi_{11}^2}{2}\, >\, 0 \hsp \text{so that} \hsp \frac{M_n}{\sqrt{n}} \cvgd \cN\left(0,\, \sigma^4\, \frac{\pi_{11}^2}{2} \right)
\end{equation}
by the same reasoning, the real and positive limit being the top-left element of $H$ (times $\sigma^2$). Since $L_n$ as defined in the statement of the theorem is the first column of $P_n$, it follows that $L_n\T (P_n\T)^{-1} = e_p\T$. Thus, by \eqref{EstErr} and noting that $v_n\, (1-\rho_n) = c$ under \ref{HypSR},
\begin{eqnarray*}
\sqrt{n\, v_n}\, \big\langle L_n,\, \whn - \theta_n \big\rangle & = & \sqrt{n\, v_n\, (1-\rho_n)}\, e_p\T \left( V_n^{\, 1/2}\, P_n^{-1}\, \frac{S_{n,\, n-1}}{n}\, (P_n^{-1})\T\, V_n^{\, 1/2} \right)^{-1} \\
& & \hsp \hsp \times \left( V_n^{\, 1/2}\, P_n^{-1} \frac{Z_n}{n}\, \right) e_p \\
& = & \sqrt{c}\, e_p\T \left[ \left( V_n^{\, 1/2}\, P_n^{-1}\, \frac{S_{n,\, n-1}}{n}\, (P_n^{-1})\T\, V_n^{\, 1/2} \right)^{-1} - H^{-1} \right] \\
& & \hsp \hsp \times \left( V_n^{\, 1/2}\, P_n^{-1} \frac{Z_n}{\sqrt{n}}\, \right) e_p \\
 & & \hsp \hsp +~ \frac{2 \sqrt{c}\, M_n}{\sigma^2\, \pi_{11}^2\, \sqrt{n}} ~ \eqdef ~ T_{1,\, n} + T_{2,\, n} \cvgd \cN\left(0,\, \frac{2 c}{\pi_{11}^2} \right)
\end{eqnarray*}
exploiting \eqref{CvgMartCompl}, the block diagonal structure of $H$ which directly gives $e_p\T H^{-1} = 2\, e_p\T/(\sigma^2 \, \pi_{11}^2)$, Slutsky's lemma and the fact that $\Vert T_{1,\, n} \Vert = o_p(\Vert T_{2,\, n} \Vert)$. The proof is now complete. \qed

\begin{rem}
\label{Rem_ProofThm}
If we replace \ref{HypUR1} by \ref{HypUR2} in Theorems \ref{Thm_Consist} and \ref{Thm_Norm}, the results still hold with slight adjustments. In particular, we have now to consider
\begin{equation*}
    V_n = \diag( 1 - \lambda_{n,\,1},\, 1 + \lambda_{n,\,2}, \, 1,\, \ldots,\, 1)
\end{equation*}
as rates, and this implies different limits in \eqref{LimWB} and \eqref{CvgSnRenorm}. The main consequence is that the top-left $1 \times 1$ block of $H_0$ is changed into a $2 \times 2$ block given by
\begin{equation*}
    \begin{pmatrix}
        \frac{\pi_{11}^2}{2} & 0 \\
        0 & \frac{\pi_{21}^2}{2}
    \end{pmatrix}.
\end{equation*}
This new submatrix is obviously real and positive definite (see Lemma \ref{Lem_ColP1} with $k=2$), and all the subsequent reasoning is perfectly similar.
\end{rem}

\section{Applications and simulations}
\label{Sec_Appli}

To illustrate the asymptotic normality of Theorem \ref{Thm_Norm}, we will use the marginal convergence stated in Corollary \ref{Cor_NormMarg}, \textit{i.e.}
\begin{equation}
\label{StatTest}
Z_n^{\, 2} \eqdef \frac{\pi_{11}^{\, 2}\, n\, v_n}{2\, c}\, \left[ \sum_{i=1}^{p} \frac{1}{(\lambda_1 \rho_n)^{\, i-1}}\, \big( \wh{\theta}_{n,\, i} - \theta_{n,\, i} \big) \right]^{\, 2} \cvgd \chi^2_1
\end{equation}
including positive and negative unit root situations ($\lambda_1 = \pm 1$). According to \ref{HypSR}, we set $\rho_n = 1-c/v_n$ with $c=1$ and $v_n = n^{\alpha}$ for some $0 < \alpha < 1$. For each simulation, $\lambda_{n,\, 2}, \ldots, \lambda_{n,\, p}$ are randomly chosen out of $[-\vert \lambda_{n,\, 1} \vert+\epsilon,\, \vert \lambda_{n,\, 1} \vert-\epsilon]$ for $\epsilon=0.1$, then the process is generated with $n=5000$ and $Z_n^{\, 2}$ is computed. The experiment is repeated 3000 times. On Figure \ref{FigHistKhi2}, the simulations are conducted with $p \in \{ 2, 3, 4 \}$ for $\alpha=1/2$, $\lambda_1 \in \{ -1, 1 \}$ and $\sigma^2 = 1$. The empirical distributions of $Z_n^{\, 2}$ are superimposed with the `true' $\chi_1^2$ density (in red), and the frequencies of observed values greater than $6.5$ are indicated. On Figure \ref{FigCurvAlphaInf}, the simulations are conducted with $p=3$, $\lambda_1 \in \{ -1, 1 \}$, $\alpha \in \{ 1/5, 1/4, 1/3, 1/2 \}$ and $\sigma^2=1$. The empirical cumulative distributions of $Z_n^{\, 2}$ are superimposed with the `true' $\chi_1^2$ cumulative distribution (in red). Figure \ref{FigCurvAlphaSup} displays the same experiments with $\alpha \in \{ 1/2, 2/3 ,3/4, 4/5 \}$. Now let us comment our observations. One can see that there are a few more outliers than expected: $\approx 2.5 \%$ of observed values greater than 6.5 whereas $\dP(\chi_1^2 > 6.5) \approx 1.1\%$. This is due to the fact that, as we have noticed, the convergence may be slow (see the perspectives in the conclusion) and the Gaussian distribution tails remain somewhat overloaded. It is particularly clear on Figure \ref{FigCurvAlphaInf} when $\alpha$ is close to 0 and, unsurprisingly, this phenomenon is widely attenuated when we choose larger values of $n$. The best results are obtained for $\alpha=1/2$ and we observe a drift which occurs faster when $\alpha$ tends to 0 (Figure \ref{FigCurvAlphaInf}) than when $\alpha$ tends to 1 (Figure \ref{FigCurvAlphaSup}). This is probably a consequence of the well-known fact that the rates of convergence are faster in an unstable model that in a stable one, as we recalled in the introduction. On the whole, we may roughly say that for the moderate values of $\alpha$ (between 1/3 and 2/3, say), the simulations are convincing and clearly fit with the theoretical behavior.

\begin{figure}[h!]
\begin{center}
\includegraphics[width=15cm]{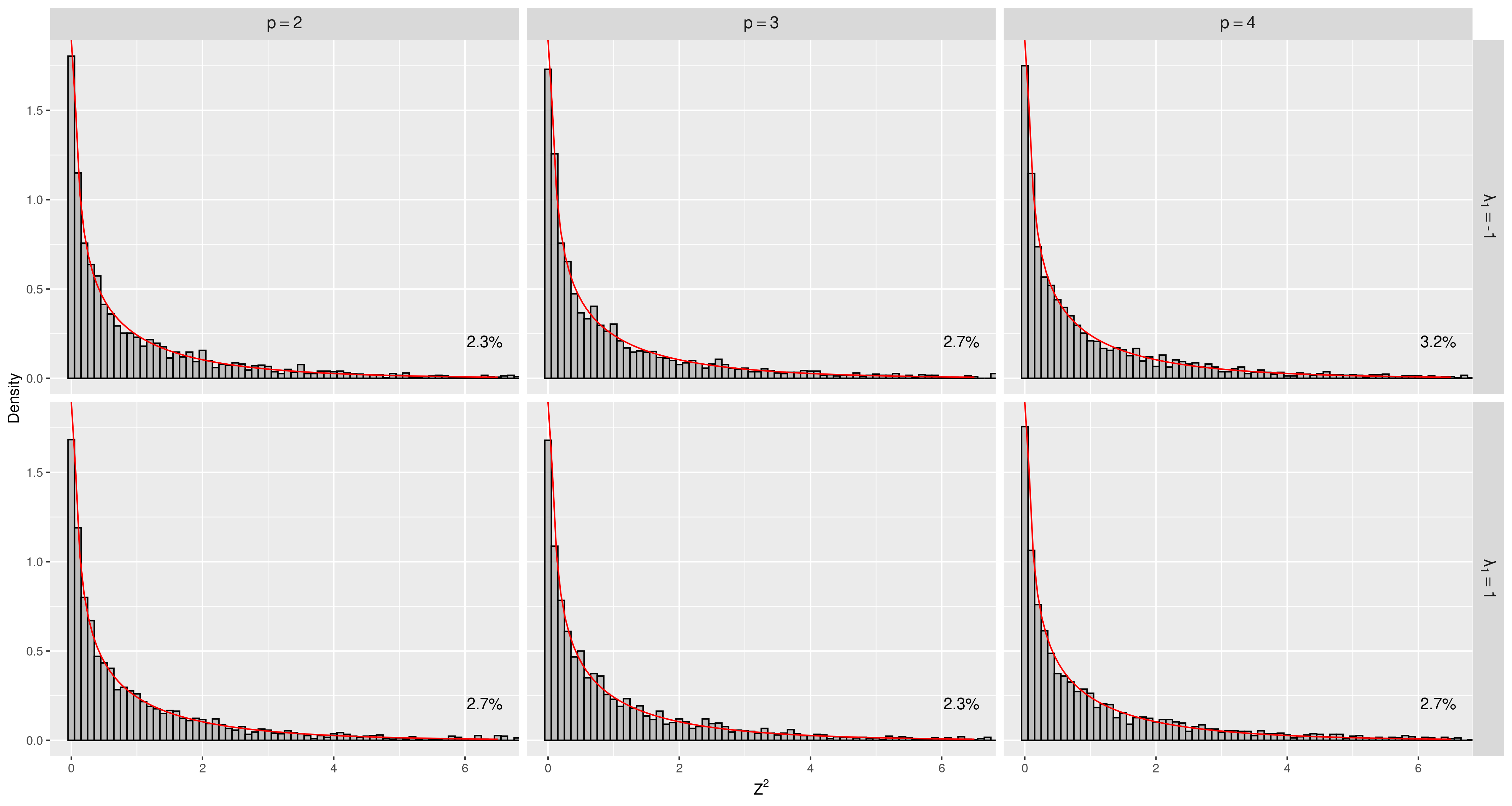}
\end{center}
\caption{Empirical distributions of $Z_n^{\, 2}$ for $p=2$ (left), $p=3$ (middle) and $p=4$ (right), with $\lambda_1 = -1$ (top) and $\lambda_1 = 1$ (bottom), on the basis of $3000$ experiments of size $n=5000$ with $\alpha=1/2$. The red curve is the $\chi_1^2$ distribution. The frequencies of observed values greater than $6.5$ are indicated, to be compared with $\dP(\chi_1^2 > 6.5) \approx 1.1\%$.}
\label{FigHistKhi2}
\end{figure}

\begin{figure}[h!]
\begin{center}
\includegraphics[width=15cm]{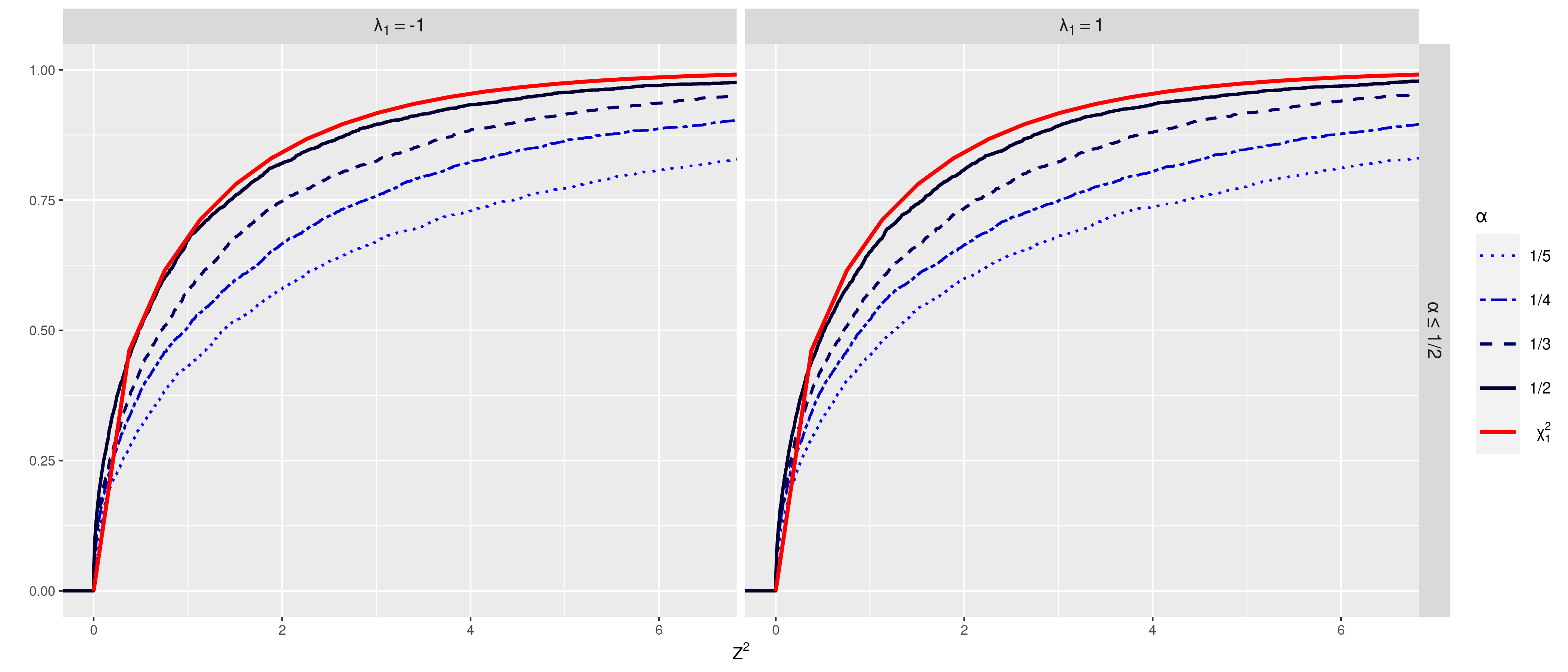}
\end{center}
\caption{Empirical cumulative distributions of $Z_n^{\, 2}$ for $p=3$ with $\lambda_1 = -1$ (left), $\lambda_1 = 1$ (right) and $\alpha \in \{1/5, 1/4, 1/3, 1/2\}$, on the basis of $3000$ experiments of size $n=5000$. The red curve is the $\chi_1^2$ cumulative distribution.}
\label{FigCurvAlphaInf}
\end{figure}

\begin{figure}[h!]
\begin{center}
\includegraphics[width=15cm]{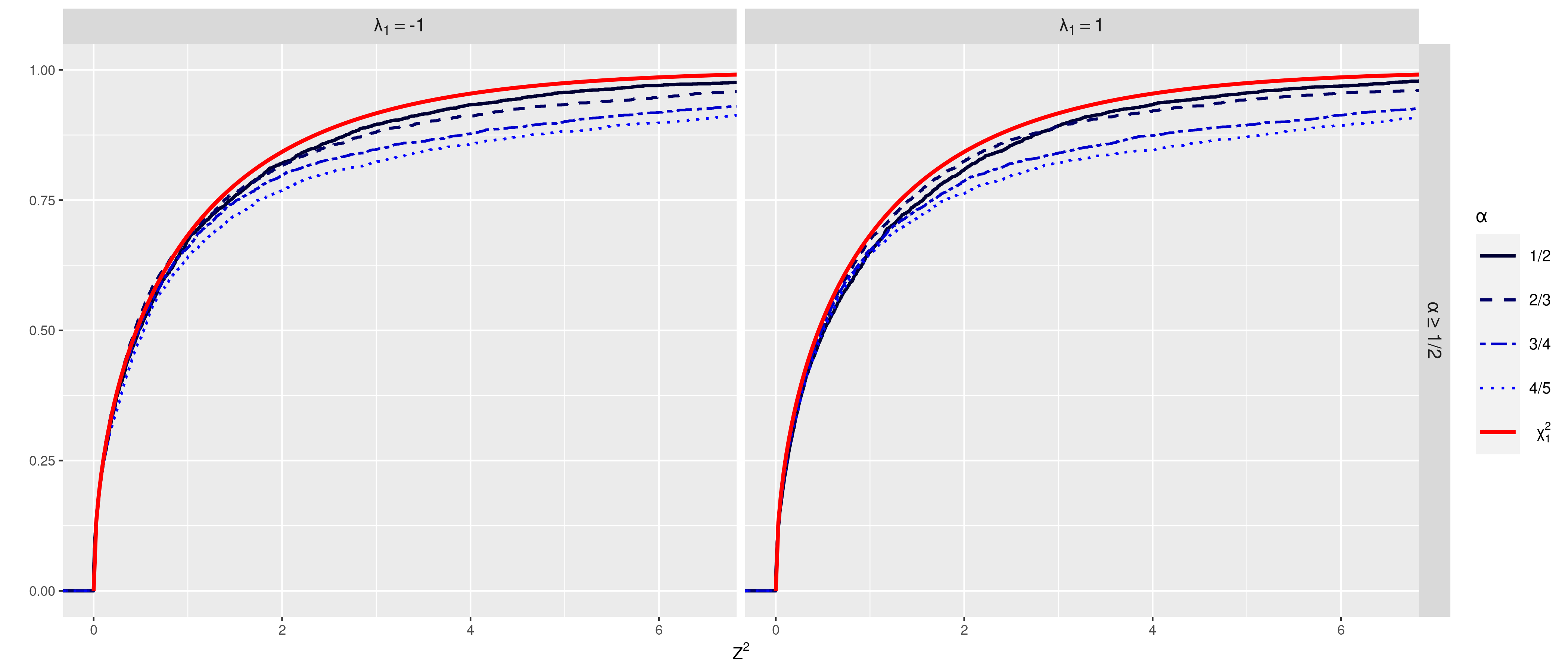}
\end{center}
\caption{Empirical cumulative distributions of $Z_n^{\, 2}$ for $p=3$ with $\lambda_1 = -1$ (left), $\lambda_1 = 1$ (right), $\alpha \in \{1/2, 2/3, 3/4, 4/5\}$, on the basis of $3000$ experiments of size $n=5000$. The red curve is the $\chi_1^2$ cumulative distribution.}
\label{FigCurvAlphaSup}
\end{figure}

\smallskip

To sum up, aligned with the work of \citet{PhillipsMagdalinos07}, this study completes the moderate deviations of \citet{Proia20} by providing a sharp analysis of the asymptotic behavior of the OLS between the rates $\sqrt{n}$ and $n$ corresponding to the stable and unstable AR$(p)$ processes, respectively, when the unit root in $A$ is either positive or negative (and even when both unit roots are present). The same conclusion prevails concerning the lack of continuity at the boundaries: $\alpha \rightarrow 0^+$ and $\alpha=0$ do not match and the same is true for $\alpha \rightarrow 1^-$ and $\alpha=1$. However, by focusing on the inner neighborhood of the unit root, nearly-unstable time-varying AR processes are instructive, especially considering the extension of the asymptotic normality when $\rho_n$ does not converge `too fast' to 1. There are still many improvements to be made, the main one being whether a unit root test can be derived. For the processes which admit a decomposition of the form
\begin{equation*}
\forall\, n \geq 1,\, \forall\, 1 \leq k \leq n, \hsp \left\{
\begin{array}{lcl}
X_{n,\, k} & = & \beta_n\, X_{n,\, k-1} + W_k \\
(W_k) & \sim & \textnormal{stable AR}(p-1),
\end{array}
\right.
\end{equation*}
an identification with \eqref{Intro_NearlyUnstableAR} is possible and shows that $(X_{n,\, k})$ is in fact a time-varying AR$(p)$ process such that if $\beta_n$ converges to $\beta$, then all its coefficients converge \textit{at the same rate}. In this subclass of processes, \citet{Park03} suggests a theoretical procedure to test for $\cH_0$ : ``$\beta = 1$" \textit{vs} $\cH_1$ : ``$\beta < 1$" provided that $\vert \beta_n - \beta \vert \asymp 1/v_n$. The authors are pretty convinced that this could be extended to all the nearly unstable processes covered by \eqref{Intro_NearlyUnstableAR} and that such a trail is likely to outperform the usual Dickey-Fuller tests for unit root. This is a work in progress. In addition, the simulations highlight an asymmetry in the convergence rates: we observed, on many more experiments than those presented here, that the asymptotic distribution deteriorates less quickly when $\alpha$ moves away from 1/2 towards 1 than towards 0 (see again Figures \ref{FigCurvAlphaInf} and \ref{FigCurvAlphaSup}). It should therefore be interesting to investigate more deeply the convergence rate of the OLS. Finally, it would be instructive but very challenging to look at the behavior of the OLS on the \textit{outer} neighborhood of the unit root, that is, for explosive but nearly unstable AR$(p)$ processes, as it is already done for $p=1$.

\smallskip

\noindent \textbf{Acknowledgements}. This research benefited from the support of the ANR project `Efficient inference for large and high-frequency data' (ANR-21-CE40-0021). The authors sincerely thank the anonymous reviewer and the associate editor for their comments and references which have clearly contributed to the improvement of the paper.

\nocite{*}
\bibliographystyle{plainnat}
\bibliography{NearlyUnstable_V2}

\end{document}